\documentclass[times,sort&compress,3p]{elsarticle}
\journal{Journal of Multivariate Analysis}
\usepackage[labelfont=bf]{caption}

\usepackage{amsmath,amsfonts,amssymb,amsthm,booktabs,color,epsfig,graphicx,hyperref,url}

\usepackage[utf8]{inputenc}

\theoremstyle{plain}

\newtheorem{lemma}{Lemma}

\newtheorem{theorem}{Theorem}

\theoremstyle{definition}
\newtheorem{definition}{Definition}
\newtheorem{remark}{Remark}

\allowdisplaybreaks

\begin{document}

\newcommand{\cB}{\mathcal{B}}
\newcommand{\cC}{\mathcal{C}}
\newcommand{\cD}{\mathcal{D}}
\newcommand{\cE}{\mathcal{E}}
\newcommand{\cF}{\mathcal{F}}
\newcommand{\cH}{\mathcal{H}}
\newcommand{\cN}{\mathcal{N}}
\newcommand{\cS}{\mathcal{S}}
\newcommand{\cL}{\mathcal{L}}
\newcommand{\cP}{\mathcal{P}}

\newcommand{\bC}{\mathbb{C}}
\newcommand{\bD}{\mathbb{D}}
\newcommand{\bE}{\mathbb{E}}
\newcommand{\bN}{\mathbb{N}}
\newcommand{\bP}{\mathbb{P}}
\newcommand{\bQ}{\mathbb{Q}}
\newcommand{\bR}{\mathbb{R}}
\newcommand{\bZ}{\mathbb{Z}}

\newcommand{\bd}{\mathrm{diag}}
\newcommand{\bT}{\mathrm{Tr}}
\newcommand{\br}{\mathrm{rank}}

\newcommand{\red}{\color{red}}
\newcommand{\blue}{\color{blue}}


\begin{frontmatter}

\title{Recent advances on eigenvalues of matrix-valued stochastic processes}

\author[1]{Jian Song}
\author[2]{Jianfeng Yao\corref{mycorrespondingauthor}}
\author[3]{Wangjun Yuan}

\address[1]{School of Mathematics, Shandong University}
\address[2]{Department of Statistics and Actuarial Science,The University of Hong Kong}
\address[3]{Department of Mathematics and Statistics, University of Ottawa}

\cortext[mycorrespondingauthor]{Corresponding author. Email address: \url{jeffyao@hku.hk}}

\begin{abstract}
Since the introduction of Dyson's Brownian motion in early 1960's, there have been a lot of developments in the investigation of stochastic processes on the space of Hermitian matrices. Their properties, especially, the properties of their eigenvalues have been studied in great details. In particular, the limiting behaviors of the eigenvalues are found when the dimension of the matrix space tends to infinity, which connects with random matrix theory.
This survey reviews a selection of results on the eigenvalues of stochastic processes from the literature of the past three decades. For most recent variations of such processes, such as matrix-valued processes driven by fractional Brownian motion or Brownian sheet, the eigenvalues of them are also discussed in this survey.
In the end, some open problems in the area are also proposed. 
\end{abstract}


\end{frontmatter}

\section{Introduction}\label{Sec:intro}

Stochastic processes with values in the space of symmetric matrices have been attracting the attention for some years.
Their introduction is commonly attributed to the celebrated work
\cite{Dyson62} by F. J. Dyson.
By that time, Gaussian matrix ensembles were well known; the distribution of their eigenvalues has a density function of the form
\begin{equation}
  G(x_1,\ldots,x_n) = C \exp\{-\beta W\}, \label{density-F}
  \quad 
	W =W(x_1,\ldots,x_n) = -\sum_{i<j} \ln|x_i-x_j| + \sum_{i}(x_i^2/2a^2), 
\end{equation}
where $\beta, a>0$ are parameters, and $C=C(\beta,a)$ is a
normalization constant.
The distribution \eqref{density-F} also appears in the { Coulomb gas} model: it is the probability  distribution of the positions of $n$ point charges which are free to move on the real line $\bR$ under the forces derived from the potential energy $W$ and in a state of thermodynamic equilibrium at a temperature $T= (k\beta)^{-1}$ ($k$ is the Boltzmann constant).
Note that the equation \eqref{density-F} is static and does not describe the evolution of the position of the point charges before reaching the equilibrium. Dyson brought in the Brownian motion to get a time-dependent model that describes the evolution of the positions from an initial distribution $F_0$. The Brownian motion, also called the { time-dependent Coulomb gas} has a simple structure, and the joint density function $F(x_1,\ldots,x_n; t)$ of the positions of the $n$ point charges at time $t>0$ is fully characterized as a solution to the Smoluchowski equation
\begin{align}
	c \frac{\partial F}{\partial t}
	= \sum_i \left[ \frac1\beta  \frac{\partial^2 F}{\partial x_i^2} - \frac{\partial}{\partial x_i}\{ E(x_i)F\} \right],
\end{align}
where $c$ is a constant and
\begin{align*}
	E(x_i) = - \frac{\partial W}{\partial x_i}
	= \sum_{j\ne i} \frac1{x_j-x_i} - \frac{x_i}{a^2}
\end{align*}
is an { external electric force}. In particular, $F(x_1,\ldots,x_n; t)$ tends to the Coulomb gas distribution $G$ in~\eqref{density-F} as $t \to \infty$.

This extension of Coulomb gas from the static equilibrium state to a dynamical version also applies to the associated Gaussian matrix ensembles.
More precisely, Dyson introduced a stochastic process with values in the space of symmetric matrices, the eigenvalues of which coincide with the dynamical Coulomb gas model.
Amazingly, this process is extremely simple:
its elements are independent Ornstein-Uhlenbeck processes on the underlying field! (The underlying field is $\bR$ for $\beta=1$, $\bC$ for $\beta=2$, and the quaternion field for $\beta=4$). 

This deep connection between stochastic processes with values in the space of symmetric (Hermitian) matrices and the induced dynamical system of its eigenvalues, had been however quite ignored during a while afterwards.
It was revived in the papers \cite{Norris1986,Bru1989} where the idea of Dyson was extended to the space of positive-definite matrices (ellipsoids).
In the subsequent three decades, the study of these stochastic processes and their associated eigenvalue processes has been developed in much depth. Particularly, the symmetric (Hermitian) matrix-valued processes have covered Brownian motion, Ornstein-Uhlenbeck process and fractional Brownian motion.

Instead of considering the $N$ particles (eigenvalues) with $N$ fixed, the limits of the empirical measures of particles when $N$ tends to infinity (high-dimensional limits) were studied in several models. In particular, the high-dimensional limit of the empirical measures of the Dyson's Brownian motion is the famous Wigner's semicircle law, which provides a dynamical version of Wigner's Theorem for GOE and GUE (see, e.g., \cite{Anderson2010}). In this aspect, the study of large particle systems is closely related to the random matrix theory. Moreover, the equation satisfied by the limits of the empirical measures of the Dyson's Brownian motion is the so-called McKean-Vlasov equation, which appears in the study of propagation of chaos for large systems of interacting particles (see \cite{Berman2019, Jabin2018, Serfaty2020}).

Another motivation for studying high-dimensional limits of the empirical measures of eigenvalues arises from free probability theory. By \cite{Biane1997}, the free additive Brownian motion can be viewed as the high-dimensional limit of a matrix Brownian motion with appropriate scaling. Moreover, \cite{Biane1997, Biane1998} developed the stochastic calculus for free Brownian motion. Besides, the non-commutative fractional Brownian motion was introduced in \cite{Nourdin2014}.

There is also a deep connection between matrix-valued  stochastic processes and
multivariate statistical analysis. Here are a few  applications of these
processes in recent statistical literature:
\begin{enumerate}
\item Financial data analysis: multivariate volatility/co-volatility (variance/covariance) between stock returns or interest rates from different markets have been studied recently through Wishart processes, see \cite{Gourieroux06,Gourieroux10a,DaFonseca08,DaFonseca14,Gnoatto12,Gnoatto14,Wu18}.
	
\item Machine learning: an important task in machine learning using kernel functions is the determination of a suitable kernel matrix for a given data analysis problem (\cite{Scholkopf02}). Such determination is	referred as the kernel matrix learning problem. A kernel matrix is in fact a positive definite Gram-matrix of size $N\times N$ where $N$, the sample size of the data, is usually	large. An innovative method for kernel learning is proposed by	\cite{Zhang06} where unknown kernel matrix is modelled by a Wishart process prior. This approach has been followed in \cite{Kondor07,Li09}.
	
\item  Computer vision: real-time computer vision often involves
  tracking of objects of interest. At each time $t$, a target is
  encoded into a $N$-dimensional vector $a_t \in \bR^N $ (feature
  vector). It is therefore	clear that measuring ``distance'' between
  these vectors, say $a_t$ and $a_{t+dt}$ at two consecutive time
  spots $t$ and	$t+dt$, is of crucial importance for object
  tracking. Because the standard Euclidean distance $\|a_{t+dt}
  -a_{t}\|^2$ is rarely optimal, it is more satisfactory to identify a
  better metric of the form $(a_{t+dt} -a_{t})^\intercal M_t (a_{t+dt}
  -a_{t})$ using a suitable positive definite matrix $M_t$.  An
  innovative model where  the process $M_t$ follows a Wishart process 
  is proposed in \cite{Li16}.
\end{enumerate}

This survey reviews a selection of results from the last three decades. In Section~\ref{Sec:Dyson}, we provide a study of  Dyson's Brownian motion with full details. This includes a modern derivation of the process using It\^{o} calculus.
A limit for the processes of empirical eigenvalue measures is derived when the number of eigenvalues, or electric charges, tends to infinity.
Besides, a limiting Gaussian process is derived in order to characterize the fluctuation of the empirical eigenvalue measures around their limit.
In Section~\ref{Sec:positive-definite matrices}, we discuss two specific classes of stochastic processes with values in the space of positive-definite matrices, that is, Brownian motions of ellipsoids and Wishart processes.
In Section \ref{Sec:otherModels}, a more general form of stochastic processes on the space of Hermitian matrices is studied, and a link is also made with some familiar systems of interacting particles.
The following Sections~\ref{Sec:fBm} and \ref{Sec:BrownianSheet} concern extensions of Dyson's Brownian motion in two different directions. The first extension replaces the Brownian motions in the matrix by { fractional} Brownian motions, and the second one by Brownian sheets. Finally in Section \ref{Sec:Open problems}, we conclude with a discussion on open problems related to the results introduced in the preceding sections.

\section{Dyson's Brownian motion} \label{Sec:Dyson}

In this section, we mainly focus on the Dyson's Brownian motion. We discuss the system of SDEs satisfied by Dyson's Brownian motion in Section \ref{Sec:DBM fix N} and the limiting behaviors of the eigenvalue empirical measure process in Section \ref{Sec:DBM limit N}.

\subsection{Finite-dimensional results} \label{Sec:DBM fix N}

Throughout the survey, we denote the complex imaginary by $\iota = \sqrt{-1}$.

\begin{definition} \label{Def-Dyson}
	Let $\{B_{i,j}(t), \tilde{B}_{i,j}(t), 1 \le i \le j \le N\}$ be a family of i.i.d. real valued standard Brownian motions. Let $H^{N,\beta}(t) = \left( H^{N,\beta}_{k,l}(t) \right)_{1 \le k \le l \le N}$ be a real symmetric ($\beta=1$) or complex Hermitian ($\beta=2$) $N \times N$ matrix-valued process with entries
	\begin{align*}
		H^{N,\beta}_{k,l}(t) =
			\dfrac{1}{\sqrt{\beta N}} \left( B_{k,l}(t) + \iota
            (\beta-1) \tilde{B}_{k,l}(t) \right) {\large 1}_{\{ k < l\}}
            + 
			\dfrac{\sqrt{2}}{\sqrt{\beta N}} B_{l,l}(t) {\large 1}_{\{ k = l\}}.
	\end{align*}
	Then $H^{N,1}(t)$ is a real symmetric matrix Brownian motion and $H^{N,2}(t)$ is a complex Hermitian matrix Brownian motion.
\end{definition}

The following results state that the eigenvalue processes of real symmetric or complex Hermitian matrix Brownian motion never collide almost surely and are characterized by a system of stochastic differential equations (SDEs).

\begin{theorem}[\cite{Anderson2010}, Theorem 4.3.2] \label{Thm-Dyson SDE}
	Let $X^{N,\beta}(0)$ be a real symmetric ($\beta=1$) or complex Hermitian ($\beta=2$) $N \times N$ deterministic matrix and let $X^{N,\beta}(t) = X^{N,\beta}(0) + H^{N,\beta}(t)$. Let $\lambda_1^{N,\beta}(t) \ge \lambda_2^{N,\beta}(t) \ge \cdots \ge \lambda_N^{N,\beta}(t)$ be the ordered eigenvalue processes of $X^{N,\beta}(t)$. Denote the first collision time of the eigenvalue processes by
	\begin{align} \label{eq-first collision time}
		\tau_{N,\beta} = \inf \left\{ t>0: \exists \ i \neq j, ~\lambda_i^{N,\beta}(t) = \lambda_j^{N,\beta}(t) \right\}.
	\end{align}
	Then
    $
		\bP \left( \tau_{N,\beta} = + \infty \right) = 1.
	$
	Furthermore, the ordered eigenvalue processes $\lambda_1^{N,\beta}(t) > \cdots > \lambda_N^{N,\beta}(t)$ are the unique solution to the following system of SDEs:
	\begin{align} \label{eq-Dyson BM SDE}
		d\lambda_i^{N,\beta}(t)
		= \dfrac{\sqrt{2}}{\sqrt{\beta N}} dW_i(t)
		+ \dfrac{1}{N} \sum_{j:j\neq i} \dfrac{dt}{\lambda_i^{N,\beta}(t) - \lambda_j^{N,\beta}(t)}, ~ i \in\{ 1, \ldots, N\}.
	\end{align}
	Here, $\{W_1(t), \ldots, W_N(t)\}$ is a family of independent standard Brownian motions.
\end{theorem}

\begin{proof}[\textbf{\upshape Proof:}]
	The proof is motivated by \cite[Theorem 3,5]{Graczyk2013} and \cite[Lemma 4.3.3]{Anderson2010}. We only consider the real symmetric case $\beta = 1$. The complex Hermitian case $\beta = 2$ is similar and thus is omitted. Since the dimension $N$ is fixed, we may omit both $N$ and $\beta$ on subscript and superscript without ambiguity. For simplicity, we only give a proof under the condition $\lambda_1(0) > \cdots > \lambda_N(0)$. For the case that $X(0)$ has collision eigenvalues, we refer the interested readers to \cite[Page 257]{Anderson2010}. We divide the proof into three steps.
	
	\vspace*{1\baselineskip}
	
	\noindent {Step 1:} Derivation of the system of SDEs for eigenvalue processes before the first collision time by It\^{o} calculus and martingale theory.
	
	We may use the Stratonovich differential notation, which can be founded in, for example, \cite[Chapter III]{Ikeda1981}. For two $N \times N$ matrices $X$ and $Y$, we have $X \circ dY = X dY + \frac{1}{2} dXdY$, where $XdY$ is the It\^{o} differential, $X \circ dY$ is the Stratonovich differential and $dXdY = d\langle X, Y \rangle$. By matrix multiplication, for three $N \times N$ matrices $X, Y$ and $Z$, one can verify that
	\begin{align*}
		& dX \circ (YZ) = (dX \circ Y) \circ Z = dX YZ + \dfrac{1}{2} \left( dXdY Z + dX Y dZ \right), \\
		& (X \circ dY) \circ Z = X \circ (dY \circ Z) = X dY Z + \dfrac{1}{2} \left( dXdY Z + X dYdZ \right), \\
		& (X \circ dY)^\intercal = dY^\intercal \circ X^\intercal.
	\end{align*}
	Moreover, by It\^{o} formula and matrix multiplication, one can verify that
	\begin{align} \label{eq-Ito for Stratonovich}
		d(XYZ) = dX \circ YZ + X \circ dY \circ Z + XY \circ dZ.
	\end{align}
	
	For a real symmetric matrix process $X(t)$, consider its spectral decomposition $X(t) = P(t) D(t) P(t)^\intercal$, where $D(t)$ is a diagonal matrix of eigenvalues of $X(t)$ ordered decreasingly, and $P(t)$ is an orthogonal matrix of eigenvectors of $X(t)$. According to \cite{Norris1986}, the matrices $D(t), P(t)$ can be chosen as smooth functions of $X(t)$ for $t < \tau_{N,\beta}$. Let $Q(t)$ be the matrix-valued processes satisfying
	\begin{align*}
		dQ(t) = P(t)^{-1} \circ dP(t) = P(t)^\intercal \circ dP(t).
	\end{align*}
	The process $Q(t)$ is known as the stochastic logarithm of $P(t)$. By the It\^{o} formula \eqref{eq-Ito for Stratonovich}, we have the following identity
	\begin{align} \label{eq-skew symmetric of Q}
		0 = dI_N = d(P(t)^\intercal P(t)) = dQ(t) + dQ(t)^\intercal.
	\end{align}
	
	Applying the It\^{o} formula \eqref{eq-Ito for Stratonovich} to the spectral decomposition of $X(t)$ and using \eqref{eq-skew symmetric of Q}, we have
	\begin{align} \label{eq-Ito for spectral of Dyson}
		dD(t) &= d P(t)^\intercal \circ X(t) P(t) + P(t)^\intercal \circ d X(t) \circ P(t) + P(t)^\intercal X(t) \circ d P(t) \nonumber \\
		&= d P(t)^\intercal \circ P(t) D(t) + P(t)^\intercal \circ d X(t) \circ P(t) + D(t) P(t)^\intercal \circ d P(t) \nonumber \\
		&= -d Q(t) \circ D(t) + P(t)^\intercal \circ d X(t) \circ P(t) + D(t) \circ d Q(t).
	\end{align}
	By considering the non-diagonal entries of \eqref{eq-Ito for spectral of Dyson}, we have
	\begin{align} \label{eq-entries of Q}
		dQ_{ij}(t) = - \left( P(t)^\intercal \circ d X(t) \circ P(t) \right)_{ij} \circ \dfrac{1}{\lambda_i(t) - \lambda_j(t)},
		~ i \neq j.
	\end{align}
	On the other hand, the diagonal entries of \eqref{eq-Ito for spectral of Dyson} can be written as
	\begin{align} \label{eq-1.6-lambda SDE}
		d\lambda_i(t) =& \left( P(t)^\intercal \circ d X(t) \circ P(t) \right)_{ii} \nonumber \\
		=& \left( P(t)^\intercal d X(t) P(t) \right)_{ii}
		+ \dfrac{1}{2} \Big( dP(t)^\intercal d X(t) P(t) + P(t)^\intercal d X(t) dP(t) \Big)_{ii}.
	\end{align}
	Recalling Definition \ref{Def-Dyson}, one can see that $\{ \left( P(t)^\intercal \circ d X(t) \circ P(t) \right)_{ii} \}_{1 \le i \le N}$ is a family of local martingales with quadratic covariation
	\begin{align*}
		& \left( P(t)^\intercal \circ d X(t) \circ P(t) \right)_{ii} \left( P(t)^\intercal \circ d X(t) \circ P(t) \right)_{jj} 
		= \left( P(t)^\intercal d X(t) P(t) \right)_{ii} \left( P(t)^\intercal d X(t) P(t) \right)_{jj} \\
		=& \sum_{k,l=1}^N P_{ki}(t) P_{li}(t) dX_{kl}(t) \sum_{k',l'=1}^N P_{k'j}(t) P_{l'j}(t) dX_{k'l'}(t)  
		= \sum_{k,l,k',l'=1}^N P_{ki}(t) P_{li}(t) P_{k'j}(t) P_{l'j}(t) \left( \text{1}_{[k=k']} \text{1}_{[l=l']} + \text{1}_{[k=l']} \text{1}_{[l=k']} \right) \dfrac{dt}{N}  \\
		=& \dfrac{2dt}{N} \left( \sum_{k=1}^N P_{ki}(t) P_{kj}(t) \right)^2
		= \dfrac{2}{N} \text{1}_{[i=j]} dt,
	\end{align*}
	where we use the orthogonality of the columns of the matrix $P(t)$. Thus, by Knight's theorem, there exists a family of independent standard $1$-dimensional Brownian motions $\{W_1(t), \ldots, W_N(t)\}$, such that
	\begin{align} \label{eq-1.7-martingale term of 1.6}
		\left( P(t)^\intercal \circ d X(t) \circ P(t) \right)_{ii}
		= \dfrac{\sqrt{2}}{\sqrt{N}} dW_i(t).
	\end{align}
	Note that $X(t)$ is symmetric, by \eqref{eq-skew symmetric of Q} and \eqref{eq-entries of Q}, we have
	\begin{align} \label{eq-1.8-dPdXP}
		& \dfrac{1}{2} \Big( dP(t)^\intercal d X(t) P(t) + P(t)^\intercal d X(t) dP(t) \Big)_{ii} \nonumber 
		= \Big( dP(t)^\intercal d X(t) P(t) \Big)_{ii} \nonumber \\
		=~& \Big( dP(t)^\intercal P(t) P(t)^\intercal d X(t) P(t) \Big)_{ii} 
		= \Big( dQ(t)^\intercal \left( P(t)^\intercal \circ d X(t) \circ P(t) \right) \Big)_{ii} \nonumber \\
		=~& -\sum_{j=1}^N dQ_{ij}(t) \left( P(t)^\intercal \circ dX(t) \circ P(t) \right)_{ji} 
		= \sum_{j:j\neq i} \dfrac{\left( P(t)^\intercal d X(t) P(t) \right)_{ij} \left( P(t)^\intercal dX(t) P(t) \right)_{ji}}{\lambda_i(t) - \lambda_j(t)}.
	\end{align}
	For $i \neq j$, we have
	\begin{align} \label{eq-1.9-Num of 1.8}
		& \left( P(t)^\intercal d X(t) P(t) \right)_{ij} \left( P(t)^\intercal d X(t) P(t) \right)_{ji} 
		= \sum_{k,l=1}^N P_{ki}(t) P_{lj}(t) dX_{kl}(t) \sum_{k',l'=1}^N P_{k'j}(t) P_{l'i}(t) dX_{k'l'}(t) \nonumber \\
		=~ & \sum_{k,l,k',l'=1}^N P_{ki}(t) P_{lj}(t) P_{k'j}(t) P_{l'i}(t) \left( \text{1}_{[k=k']} \text{1}_{[l=l']} + \text{1}_{[k=l']} \text{1}_{[l=k']} \right) \dfrac{dt}{N} \nonumber \\
		=~ & \dfrac{dt}{N} \left( \sum_{k=1}^N P_{ki}(t) P_{kj}(t) \right)^2
		+ \dfrac{dt}{N} \left( \sum_{k=1}^N P_{ki}(t)^2 \right) \left( \sum_{l=1}^N P_{lj}(t)^2 \right) 		= \dfrac{dt}{N}.
	\end{align}
	Substituting \eqref{eq-1.9-Num of 1.8} to \eqref{eq-1.8-dPdXP}, we have
	\begin{align} \label{eq-1.10-finite variation term of 1.6}
		\dfrac{1}{2} \Big( dP(t)^\intercal d X(t) P(t) + P(t)^\intercal d X(t) dP(t) \Big)_{ii}
		= \dfrac{1}{N} \sum_{j:j\neq i} \dfrac{dt}{\lambda_i(t) - \lambda_j(t)}.
	\end{align}
	Therefore, \eqref{eq-Dyson BM SDE} follows from \eqref{eq-1.6-lambda SDE}, \eqref{eq-1.7-martingale term of 1.6} and \eqref{eq-1.10-finite variation term of 1.6}.
	
	\vspace*{1\baselineskip}
	
	\noindent {Step 2:} We prove that the system of SDEs \eqref{eq-Dyson BM SDE} has a unique strong solution before its first collision time by approximating the singular drift with regular functions. For the existence and uniqueness of SDE, we refer to \cite{Karatzas1998}.
	
	For $R > 0$, define
	\begin{align} \label{eq-def-psi}
		\psi_R(x) =
			x^{-1}  {\large 1}_{\{ |x| \ge R^{-1}\}}
		    + 	R^2 x  {\large 1}_{\{ |x| < R^{-1}\}}.
	\end{align}
	One can easily check that $\psi_R(x)$ is continuous on $\bR$ satisfying
    $		|\psi_R(x)| \le (1+R^2) (1+|x|^2).
    $
	Consider the following system of SDEs
	\begin{align} \label{eq-1.11-Dyson SDE approximation}
		d\lambda_i^R(t)
		= \dfrac{\sqrt{2}}{\sqrt{N}} dW_i(t)
		+ \dfrac{1}{N} \sum_{j:j\neq i} \psi_R \left( \lambda_i^R(t) -
        \lambda_j^R(t) \right) dt, ~ i \in \{ 1, \ldots, N\},
	\end{align}
	with initial condition $\lambda_i^R(0) = \lambda_i(0)$ for $1 \le i \le N$. Noting that for each $R > 0$, the coefficient functions in \eqref{eq-1.11-Dyson SDE approximation} are global Lipschitz and of linear growth, the existence of the strong solution of \eqref{eq-1.11-Dyson SDE approximation} follows from \cite[Theorem 2.9]{Karatzas1998}, and moreover, by \cite[Theorem 2.5]{Karatzas1998}, we also have the strong uniqueness.
	
	For $R>0$, let
	\begin{align*}
		\tau(R) = \inf \left\{ t>0: \min_{i \neq j} \left| \lambda_i^R(t) - \lambda_j^R(t) \right| < R^{-1} \right\}.
	\end{align*}
	Then $\tau(R)$ is a stopping time which is increasing with respect to $R$. We denote $\tau(+\infty) = \lim_{R \to +\infty} \tau(R)$, which may be $+\infty$. Let $R_0$ be a positive number such that $R_0^{-1} = \min_{i \neq j} |\lambda_i(0) - \lambda_j(0)|$. For $R_1 > R_2 > R_0$, we have the following observation
	\begin{align*}
		\lambda_i^{R_1}(t) = \lambda_i^{R_2}(t), ~ \forall t \le \tau(R_2), ~ \forall 1 \le i \le N.
	\end{align*}
	Thus, for $t < \tau(+\infty)$, we can define the processes $\lambda_i^{\infty}(t)$ in a consistent way by
	\begin{align*}
		\lambda_i^{\infty}(t) = \lambda_i^R(t), ~\text{if} ~ t < \tau(R)
	\end{align*}
	for $1 \le i \le N$. Then, recalling the definition \eqref{eq-def-psi} of $\psi_R$, $(\lambda_1^{\infty}(t), \ldots, \lambda_N^{\infty}(t))$ solves \eqref{eq-Dyson BM SDE} for $t < \tau(+\infty)$. Note that for any strong solution of \eqref{eq-Dyson BM SDE}, it solves \eqref{eq-1.11-Dyson SDE approximation} before the time when the least distance of its entries reaches $R^{-1}$ for $R > R_0$. Thus, the strong uniqueness of \eqref{eq-Dyson BM SDE} follows from the strong uniqueness of \eqref{eq-1.11-Dyson SDE approximation} by letting $R \to \infty$.
	
	\vspace*{1\baselineskip}
	
	\noindent {Step 3:} We prove the almost sure non-collision of the eigenvalue processes by McKean’s argument (\cite[Proposition 4.3]{Mayerhofer2011}, see also \cite{McKean}).
	
	From Step 1 and Step 2, we can see that the eigenvalue processes of $H^{N,\beta}(t)$ is the unique strong solution to \eqref{eq-Dyson BM SDE}, and thus $\tau_{N,\beta}$ given by \eqref{eq-first collision time} is also the collision time for the strong solution to \eqref{eq-Dyson BM SDE}. For $t < \tau_{N,\beta}$, define
	\begin{align} \label{eq-U for noncollision}
		U(t) = \sum_{i<j} \ln |\lambda_i(t) - \lambda_j(t)|,
	\end{align}
	then by \eqref{eq-Dyson BM SDE} and It\^{o} formula, noting that $d\langle \lambda_i(t), \lambda_j(t) \rangle = 0$ for $i \neq j$, we have
	\begin{align}
		dU(t) =& \sum_{i \neq j} \dfrac{d\lambda_i(t)}{\lambda_i(t) - \lambda_j(t)} - \dfrac{1}{2} \sum_{i \neq j} \dfrac{d\langle \lambda_i(t) \rangle}{\left( \lambda_i(t) - \lambda_j(t) \right)^2}  \\
		=& \dfrac{\sqrt{2}}{\sqrt{N}} \sum_{i \neq j} \dfrac{dW_i}{\lambda_i(t) - \lambda_j(t)}
		+ \dfrac{1}{N} \sum_{i \neq j} \sum_{l:l \neq i} \dfrac{dt}{\left( \lambda_i(t) - \lambda_j(t) \right) \left( \lambda_i(t) - \lambda_l(t) \right)} 
		- \dfrac{1}{N} \sum_{i \neq j} \dfrac{dt}{\left( \lambda_i(t) - \lambda_j(t) \right)^2} \nonumber \\
		=& \dfrac{\sqrt{2}}{\sqrt{N}} \sum_{i \neq j} \dfrac{dW_i}{\lambda_i(t) - \lambda_j(t)}
		+ \dfrac{1}{N} \sum_{i \neq j \neq l \neq i} \dfrac{dt}{\left( \lambda_i(t) - \lambda_j(t) \right) \left( \lambda_i(t) - \lambda_l(t) \right)} \nonumber \\
		=& \dfrac{\sqrt{2}}{\sqrt{N}} \sum_{i \neq j} \dfrac{dW_i}{\lambda_i(t) - \lambda_j(t)}
		+ \dfrac{1}{N} \sum_{i \neq j \neq l \neq i} \dfrac{\lambda_l(t) - \lambda_j(t)}{\left( \lambda_i(t) - \lambda_j(t) \right) \left( \lambda_j(t) - \lambda_l(t) \right) \left( \lambda_l(t) - \lambda_i(t) \right)} dt \nonumber \\
		=& \dfrac{\sqrt{2}}{\sqrt{N}} \sum_{i \neq j} \dfrac{dW_i}{\lambda_i(t) - \lambda_j(t)} 
		+ \dfrac{1}{3N} \sum_{i \neq j \neq l \neq i} \dfrac{\left( \lambda_l(t) - \lambda_j(t) \right) + \left( \lambda_i(t) - \lambda_l(t) \right) + \left( \lambda_j(t) - \lambda_i(t) \right)}{\left( \lambda_i(t) - \lambda_j(t) \right) \left( \lambda_j(t) - \lambda_l(t) \right) \left( \lambda_l(t) - \lambda_i(t) \right)} dt 
		= \dfrac{\sqrt{2}}{\sqrt{N}} \sum_{i \neq j} \dfrac{dW_i}{\lambda_i(t) - \lambda_j(t)}.\nonumber
	\end{align}
	Here, we use the symmetry to change the summation index in the fifth equality. Therefore, by Lemma \ref{Lemma-McKean argument} below, we have $\tau_{N,\beta} = + \infty$ almost surely.
	The proof of Theorem \ref{Thm-Dyson SDE} is complete.
\end{proof}

The following lemma is used in the Step 3 of the proof and is known as the McKean’s argument, which can be found in \cite{Mayerhofer2011}.

\begin{lemma}[McKean’s argument, \cite{Mayerhofer2011}, Proposition 4.3]
  \label{Lemma-McKean argument}
  Let $Z = \{Z(t); 0 \le t < \infty\}$ be an adapted $\bR_+$-valued
  stochastic process that is right-continuous with finite left-hand
  limits (RCLL) on a stochastic interval $[0,\tau_0)$ with $Z_0>0$,
  where
  \[  \tau_0 = \inf \{ s>0: Z_{s-} = 0 \}.  \]
  Suppose that there exists a continuous function $h$ satisfying the following:
  \begin{enumerate}
  \item[{\em (i)}] For all $t \in [0,\tau_0)$, we have
    $	  h(Z(t)) = h(Z(0)) + M(t) + P(t)$,
		where $M$ is a continuous local martingale on $[0,\tau_0)$
          with $M(0) = 0$, and $P$ is an adapted RCLL process on
          $[0,\tau_0)$ such that almost surely and  for each $T>0$,
            $\displaystyle
			\inf_{t \in [0,\tau_0 \wedge T)} P(t) > - \infty.$
		\item [{\em (ii)}] $\lim_{z \downarrow 0} h(z) = - \infty$.
	\end{enumerate}
	Then $\tau_0 = \infty$ almost surely.
\end{lemma} 

\begin{remark}
	The argument for non-collision in \cite{Anderson2010} is different. For $M>0$, it is shown that the first time for $U(t)$ with $\lambda_i(t)$ replaced by the $\lambda_i^R(t)$ to exceed $M$ is greater than any positive number almost surely via Markov inequality and Borel–Cantelli Lemma.
\end{remark}

\begin{remark}
	The unique solution to \eqref{eq-Dyson BM SDE} is known as Dyson Brownian motion.
\end{remark}

\begin{remark}
	The process given in \eqref{eq-Dyson BM SDE} with general $\beta \in (0,\infty)$ is known as $\beta$-Dyson Brownian motion. By the same argument used in Step 2, one can show that \eqref{eq-Dyson BM SDE} with general $\beta \in (0,\infty)$ has a unique strong solution before the first collision time. Moreover, applying It\^{o} formula to $U(t)$ given in \eqref{eq-U for noncollision}, we have
	\begin{align*}
		dU(t) = \dfrac{\sqrt{2}}{\sqrt{\beta N}} \sum_{i \neq j} \dfrac{dW_i}{\lambda_i^{N,\beta}(t) - \lambda_j^{N,\beta}(t)}
		+ \left( 1 - \dfrac{1}{\beta} \right) \dfrac{1}{N} \sum_{i \neq j} \dfrac{dt}{\big( \lambda_i^{N,\beta}(t) - \lambda_j(t)^{N,\beta} \big)^2}.
	\end{align*}
	Then the non-collision of the system of particles $\lambda_1^{N,\beta}(t) \ge \cdots \ge \lambda_N^{N,\beta}(t)$ follows from McKean's argument for the case $\beta \ge 1$. It is well known that the $\beta$-Dyson Brownian motion has collisions for $\beta \in (0,1)$ (see, e.g., \cite[Remark 3]{Graczyk2013}).
\end{remark}

Real symmetric matrix whose entries are i.i.d. Ornstein-Uhlenbeck processes (real symmetric matrix OU process) was considered in \cite{Chan1992}. Let $X^N(t)$ be a symmetric $N \times N$ matrix-valued process that solves the following matrix SDE
\begin{align} \label{eq-matrix SDE OU}
	dX^N(t) = \dfrac{1}{2\sqrt{N}} \left( dB(t) + dB(t)^\intercal \right) - \dfrac{1}{2} X^N(t) dt,
\end{align}
where $B(t)$ is a $N \times N$ matrix Brownian motion. Then the entries $\{X^N_{i,j}(t)\}_{1 \le i \le j \le N}$ are independent Ornstein-Uhlenbeck processes with invariant distribution $N(0,(1+\delta_{ij})/(2N))$. By It\^{o} calculus and martingale theory, \cite{Chan1992} derived the following system of SDEs for the eigenvalue processes $\{\lambda_i^N(t)\}_{1 \le i \le N}$ of $X^N(t)$ in \eqref{eq-matrix SDE OU}
\begin{align} \label{eq-eigenvalue SDE matrix OU}
	d\lambda_i^N(t) = \dfrac{1}{\sqrt{N}} dB_i(t) + \left( -\dfrac{1}{2} \lambda_i^N(t) + \dfrac{1}{2N} \sum_{j:j\neq i} \dfrac{1}{\lambda_i^N(t) - \lambda_j^N(t)} \right) dt,
	~ 1 \le i \le N.
\end{align}
By assuming the non-collision of the initial state $\lambda_1^N(0) > \cdots > \lambda_N^N(0)$, the non-collision of the eigenvalue processes was also established in \cite{Chan1992} by an argument similar to the one used in the proof of Theorem \ref{Thm-Dyson SDE}.

\subsection{High-dimensional limits} \label{Sec:DBM limit N}

Let $\cP(\bR)$ be the space of probability measures on $\bR$ equipped with the weak topology and corresponding metric $d_{\cP(\bR)}$. For $T>0$, let $C([0,T],\cP(\bR))$ be the space of continuous processes with values in $\cP(\bR)$. Then the space $C([0,T],\cP(\bR))$ endowed with the metric
\begin{align*}
	d_{C([0,T],\cP(\bR))} \left( \mu^{(1)}, \mu^{(2)} \right)
	= \sup_{t\in [0,T]} d_{\cP(\bR)} \left( \mu^{(1)}(t), \mu^{(2)}(t) \right),
\end{align*}
is complete. For a test function $f(x)$ and a measure $\mu(dx)$ on $\bR$, we write
$	\langle f, \mu \rangle = \int_{\bR} f(x) \mu(dx).
$

Recall the definition of $X^{N,\beta}(t)$ in Theorem \ref{Thm-Dyson SDE}. Let $L_N^{\beta}(t)$ be the empirical measure of the eigenvalue processes $\{\lambda_i^{N,\beta}(t)\}_ {1\le i\le N}$ of $X^{N,\beta}(t)$, that is 
\begin{align} \label{eq-empirical eigenvalue}
	L_N^{\beta}(t) (dx) = \dfrac{1}{N} \sum_{i=1}^N \delta_{\lambda_i^{N,\beta}(t)} (dx).
\end{align}
In connection with the theory of random matrices, it is of interest to investigate possible limits of these empirical measures $\{L_N^{\beta}(t), t\in[0,T]\}_{N \in \bN}$ when $N$ grows to infinity.

Such high-dimensional limits are known in the literature only for some simple cases. An early result for eigenvalue empirical measure processes can be found in \cite{Chan1992}: the exponential tightness of the sequence of corresponding eigenvalue empirical measure processes was established, which implies the almost sure convergence of the sequence. The equation satisfied by the limiting measure-valued process was also obtained, which is known as McKean-Vlasov equation. Moreover, \cite{Chan1992} proved that the semi-circle law is the only equilibrium point of the equation (with finite moments of all orders).

The high-dimensional limit results were later generalized in \cite{Rogers1993} to the following system of symmetric matrix SDE
\begin{align} \label{eq-matrix SDE general OU}
	dX^N(t) = \sqrt{\dfrac{\alpha}{2N}} \left( dB(t) + dB(t)^\intercal \right) - \theta X^N(t) dt,
\end{align}
where $B(t)$ is a $N \times N$ matrix Brownian motion. Note that if we choose $\alpha = 1$ and $\theta = 0$, then the $X^N(t)$ in \eqref{eq-matrix SDE general OU} is the real symmetric matrix Brownian motion appeared in Theorem \ref{Thm-Dyson SDE}. The real symmetric matrix OU processes in \eqref{eq-matrix SDE OU} corresponds to the case $\alpha = 1/2$ and $\theta = 1/2$. The eigenvalue processes $\{\lambda_i^N(t)\}_{1 \le i \le N}$ of $X^N(t)$ in \eqref{eq-matrix SDE general OU} are called the interacting Brownian particles in \cite{Rogers1993} and satisfy the following system of SDEs
\begin{align} \label{eq-SDE particle Rogers}
	d\lambda_i^N(t) = \sqrt{\dfrac{2\alpha}{N}} dB_i(t) + \left( - \theta \lambda_i^N(t) + \dfrac{\alpha}{N} \sum_{j:j\neq i} \dfrac{1}{\lambda_i^N(t) - \lambda_j^N(t)} \right) dt, \quad 1 \le i \le N, ~~t \ge 0.
\end{align}
In \cite{Rogers1993}, the non-collision and non-explosion of the particles \eqref{eq-SDE particle Rogers} was established assuming initial state $\lambda_1^N(0) > \cdots > \lambda_N^N(0)$. Moreover, \cite[Theorem 1]{Rogers1993} proved the weak convergence in law of the sequence of eigenvalue empirical measure processes by It\^{o} calculus and a tightness argument that is similar to Theorem \ref{Thm-Wigner dynamic-weak}. The equation that characterizes the limiting measure valued process was also derived.

The family of eigenvalue processes given in \eqref{eq-SDE particle Rogers} was further generalized in \cite{Cepa1997}. More precisely, for some Lipschitz functions $b_N$, $\sigma_N$ and positive constant $\gamma_N$, \cite{Cepa1997} proved that the following particle system
\begin{align} \label{eq-SDE particle Cepa97}
	dx_i^N(t) = \sigma_N(x_i^N(t)) dB_i(t) + \left( b_N(x_i^N(t)) + \sum_{j:j\neq i} \dfrac{\gamma_N}{x_i^N(t) - x_j^N(t)} \right) dt, \quad 1 \le i \le N, ~~t \ge 0,
\end{align}
has a unique strong solution for all the time, even with collision. For the case $\gamma_N = 2\gamma/N$, \cite[Theorem 4.2]{Cepa1997} established the weak convergence in law of the sequence of eigenvalue empirical measure processes and derived the equation for all possible limits. In \cite[Theorem 5.1]{Cepa1997}, the uniqueness of this equation was obtained if $b_N(x)$ is linear and $\sigma_N(x) = \sigma_N > 0$. The non-collision property was also established in \cite[Proposition 4.1]{Cepa1997} under the assumptions that the particles are distinct at $t=0$, $b_N(x)$ is linear, and $\sigma_N(x) = \sigma_N \in [0,\sqrt{2\gamma_N}]$. However, it is worth pointing out that the high-dimensional results does not require the non-collision of the particles.

Another generalization of the real symmetric matrix Brownian motion in Theorem \ref{Thm-Dyson SDE} and real symmetric matrix OU process in \eqref{eq-matrix SDE OU} was introduced in \cite{LLX20} as the solution of the following matrix SDE
\begin{align} \label{eq-matrix SDE Li Xiang-Dong}
	dX^N(t) = \dfrac{1}{\sqrt{2N}} \left( dB(t) + dB(t)^\intercal \right) - \dfrac{1}{2} V'(X^N(t)) dt,
\end{align}
whose ordered eigenvalue processes $\{\lambda_i^N(t)\}_{1 \le i \le N}$ satisfy
\begin{align} \label{eq-SDE Li Xiang-Dong}
	d\lambda_i^N(t) = \sqrt{\dfrac{2}{N}} dB_i(t) + \left( - \dfrac{1}{2} V'\left( \lambda_i^N(t) \right) + \dfrac{1}{N} \sum_{j:j\neq i} \dfrac{1}{\lambda_i^N(t) - \lambda_j^N(t)} \right) dt, \quad 1 \le i \le N, ~t \ge 0.
\end{align}
Here, $V$ is an external potential functions in $C^1(\bR)$ satisfying certain convexity conditions. The weak convergence in law of the sequence of eigenvalue empirical measure processes and the equation for the limiting process were obtained in \cite[Theorem 1.1]{LLX20}.

For the real symmetric or complex Hermitian matrix Brownian motion $X^{N,\beta}$ defined in Theorem \ref{Thm-Dyson SDE} with null initial value $X^{N,\beta}(0) = 0$, the high-dimensional limits was investigated in \cite{Duvillard2001} by studying large deviation bounds. The exponential tightness of the sequence $\{L_N^{\beta}(t)\}_{N \in \bN}$ was established. In \cite[Corollary 1.2]{Duvillard2001}, the almost sure convergence of the sequence $\{L_N^{\beta}(t)\}_{N \in \bN}$ was obtained and the equation for the limit was derived. Moreover, the limit was proved to be the semi-circular law. The complex case was also studied in \cite[Proposition 3.1]{pt07} where the convergence in probability was obtained. We present \cite[Proposition 4.3.10]{Anderson2010} below, where the high-dimensional limit of the sequence $\{L_N^{\beta}(t)\}_{N \in \bN}$ was recovered without assuming the null initial condition.

\begin{theorem} \label{Thm-Wigner dynamic-strong}
	Let $T>0$ be a fixed number. Suppose that there exists a positive function $\varphi \in C^2(\bR)$ with bounded first and second derivatives and satisfying $\lim_{|x|\to\infty} \varphi(x) = +\infty$, such that
$\displaystyle
		C_0 := \sup_{N\in\bN} \langle \varphi, L_N^{\beta}(0) \rangle < \infty.
$
	Assume that $L_N^{\beta}(0)$ converges weakly as $N$ tends to infinity towards a probability measure $\mu_0$.

	Then the sequence $\{L_N^{\beta}(t), t\in[0,T]\}_{N \in \bN}$
    converges {almost surely} in $C([0,T], \cP(\bR))$. Its limit
    $\mu$ is characterized by the following equation: 	for any $f \in C_b^2(\bR)$,
	\begin{align} \label{eq-equation for Dyson limit measure}
		\langle f, \mu_t \rangle
		= \langle f, \mu_0 \rangle
		+ \dfrac{1}{2} \int_{0}^t \iint_{\bR^2} \dfrac{f'(x) - f'(y)}{x-y} \mu_s(dx) \mu_s(dy) ds, \quad \forall t \in [0,T].
	\end{align}

\end{theorem}

\begin{proof}[\textbf{\upshape Proof:}]
	The idea of the proof comes from \cite[Proposition 4.3.10]{Anderson2010} (see also \cite{Song2019}). We divide the proof into four steps.
	
	\vspace*{1\baselineskip}
	
	\noindent {Step 1:} (Computation of $\langle f, L_N^{\beta}(t) \rangle$ by It\^{o} calculus.)~
	By the definition \eqref{eq-empirical eigenvalue} of $L_N^{\beta}(t)$, for $f \in C^2(\bR)$,
	\begin{align*}
		\langle f, L_N^{\beta}(t) \rangle
		= \int f(x) L_N^{\beta}(t)(dx)
		= \dfrac{1}{N} \sum_{i=1}^N \int f(x) \delta_{\lambda_i^{N,\beta}(t)}(dx)
		= \dfrac{1}{N} \sum_{i=1}^N f(\lambda_i^{N,\beta}(t)).
	\end{align*}
	By It\^{o}'s formula and \eqref{eq-Dyson BM SDE},
	\begin{align*}
		f(\lambda_i^{N,\beta}(t))
		&= f(\lambda_i^{N,\beta}(0)) + \int_{0}^t f'(\lambda_i^{N,\beta}(s)) d\lambda_i^{N,\beta}(s) + \dfrac{1}{2} \int_{0}^t f''(\lambda_i^{N,\beta}(s)) d \langle \lambda_i^{N,\beta} \rangle_s \\
		&= f(\lambda_i^{N,\beta}(0)) + \dfrac{\sqrt{2}}{\sqrt{\beta N}} \int_{0}^t f'(\lambda_i^{N,\beta}(s)) dW_i(s)
		+ \dfrac{1}{\beta N} \int_{0}^t f''(\lambda_i^{N,\beta}(s)) ds
		\\
        & \qquad + \dfrac{1}{N} \int_{0}^t f'(\lambda_i^{N,\beta}(s)) \sum_{j:j\neq i} \dfrac{1}{\lambda_i^{N,\beta}(s) - \lambda_j^{N,\beta}(s)} ds.
	\end{align*}
	Thus, using the convention $\frac{f'(x) - f'(y)}{x-y} = f''(x)$ on $\{x=y\}$, we have
	\begin{align} \label{eq-Ito formula}
		\langle f, L_N^{\beta}(t) \rangle
		=& \dfrac{1}{N} \sum_{i=1}^N f(\lambda_i^{N,\beta}(0))
		+ \dfrac{\sqrt{2}}{\sqrt{\beta N^3}} \sum_{i=1}^N \int_{0}^t f'(\lambda_i^{N,\beta}(s)) dW_i(s)
		+ \dfrac{1}{\beta N^2} \sum_{i=1}^N \int_{0}^t f''(\lambda_i^{N,\beta}(s)) ds \nonumber \\
		&+ \dfrac{1}{N^2} \sum_{i\neq j} \int_{0}^t \dfrac{f'(\lambda_i^{N,\beta}(s))}{\lambda_i^{N,\beta}(s) - \lambda_j^{N,\beta}(s)} ds \nonumber \\
		=& \dfrac{1}{N} \sum_{i=1}^N f(\lambda_i^{N,\beta}(0))
		+ \dfrac{\sqrt{2}}{\sqrt{\beta N^3}} \sum_{i=1}^N \int_{0}^t f'(\lambda_i^{N,\beta}(s)) dW_i(s)
		+ \dfrac{1}{\beta N^2} \sum_{i=1}^N \int_{0}^t f''(\lambda_i^{N,\beta}(s)) ds \nonumber \\
		&+ \dfrac{1}{2N^2} \sum_{i\neq j} \int_{0}^t \dfrac{f'(\lambda_i^{N,\beta}(s)) - f'(\lambda_j^{N,\beta}(s))}{\lambda_i^{N,\beta}(s) - \lambda_j^{N,\beta}(s)} ds \nonumber \\
		=& \dfrac{1}{N} \sum_{i=1}^N f(\lambda_i^{N,\beta}(0))
		+ \dfrac{\sqrt{2}}{\sqrt{\beta N^3}} \sum_{i=1}^N \int_{0}^t f'(\lambda_i^{N,\beta}(s)) dW_i(s) \nonumber \\
		&+ \left( \dfrac{1}{\beta} - \dfrac{1}{2} \right) \dfrac{1}{N^2} \sum_{i=1}^N \int_{0}^t f''(\lambda_i^{N,\beta}(s)) ds
		+ \dfrac{1}{2N^2} \sum_{i,j=1}^N \int_{0}^t \dfrac{f'(\lambda_i^{N,\beta}(s)) - f'(\lambda_j^{N,\beta}(s))}{\lambda_i^{N,\beta}(s) - \lambda_j^{N,\beta}(s)} ds \nonumber \\
		=& \langle f, L_N^{\beta}(0) \rangle
		+ \dfrac{\sqrt{2}}{\sqrt{\beta N^3}} \sum_{i=1}^N \int_{0}^t f'(\lambda_i^{N,\beta}(s)) dW_i(s)
		+ \left( \dfrac{1}{\beta} - \dfrac{1}{2} \right) \dfrac{1}{N} \int_0^t \langle f'', L_N^{\beta}(s) \rangle ds \nonumber \\
		& + \dfrac{1}{2} \int_{0}^t \iint_{\bR^2} \dfrac{f'(x) - f'(y)}{x-y} L_N^{\beta}(s)(dx) L_N^{\beta}(s)(dy) ds.
	\end{align}
	
	\vspace*{1\baselineskip}
	
	\noindent {Step 2:} We prove that the sequence $\{L_N^{\beta}(t), t\in[0,T]\}_{N \in \bN}$ is almost surely relatively compact in $C([0,T], \cP(\bR))$, that is, every subsequence has a further subsequence that converges in $C([0,T], \cP(\bR))$ almost surely, following the argument \cite[Lemma 4.3.13]{Anderson2010}.
	
	Note that for $f \in C^2(\bR)$ with bounded first and second derivatives, by mean value theorem, one can show $|\frac{f'(x) - f'(y)}{x-y}| \le \|f''\|_{L^{\infty}}$. Hence, by \eqref{eq-Ito formula},
	\begin{align} \label{eq-1.15-Holder estimate}
		|\langle f, L_N^{\beta}(t) \rangle - \langle f, L_N^{\beta}(s) \rangle| 
		\le~& \left| \dfrac{\sqrt{2}}{\sqrt{\beta N^3}} \sum_{i=1}^N \int_s^t f'(\lambda_i^{N,\beta}(r)) dW_i(r) \right|
		+ \left( \dfrac{1}{\beta} - \dfrac{1}{2} \right) \dfrac{1}{N} \left| \int_s^t \langle f'', L_N^{\beta}(r) \rangle dr \right| \nonumber \\
		&~ + \dfrac{1}{2} \left| \int_s^t \iint_{\bR^2} \dfrac{f'(x) - f'(y)}{x-y} L_N^{\beta}(r)(dx) L_N^{\beta}(r)(dy) dr \right| \nonumber \\
		\le~& \left| \dfrac{\sqrt{2}}{\sqrt{\beta N^3}} \sum_{i=1}^N \int_s^t f'(\lambda_i^{N,\beta}(r)) dW_i(r) \right|
		+ \left( \dfrac{1}{2} + \dfrac{1}{\beta N} - \dfrac{1}{2N} \right) \|f''\|_{L^{\infty}} |t-s|.
	\end{align}
	Note that $[0,T]$ can be partitioned into small intervals of length $\eta < \|f''\|_{L^{\infty}}^{-8/7}$ and the number of the intervals is $J = [T\eta^{-1}]$. Then by  Markov inequality and Burkholder-Davis-Gundy inequality, we have, for $M>0$,
	\begin{align} \label{eq-1.16-martingale term}
		& \mathbb{P} \left( \sup_{|t-s|\le \eta} \left| \dfrac{\sqrt{2}}{\sqrt{\beta N^3}} \sum_{i=1}^N \int_s^t f'(\lambda_i^{N,\beta}(r)) dW_i(r) \right| \ge M\eta^{1/8} \right) 
		\le \sum_{k=0}^{J-1} \mathbb{P} \left( \sup_{k\eta \le t \le (k+1)\eta} \left| \dfrac{\sqrt{2}}{\sqrt{\beta N^3}} \sum_{i=1}^N \int_{k\eta}^t f'(\lambda_i^{N,\beta}(r)) dW_i(r) \right| \ge \dfrac{M\eta^{1/8}}{3} \right) \nonumber \\
		\le& \sum_{k=0}^{J-1} \dfrac{81}{M^4 \eta^{1/2}} \mathbb{E} \left[ \sup_{k\eta \le t \le (k+1)\eta} \left| \dfrac{\sqrt{2}}{\sqrt{\beta N^3}} \sum_{i=1}^N \int_{k\eta}^t f'(\lambda_i^{N,\beta}(r)) dW_i(r) \right|^4 \right] 
		\le \sum_{k=0}^{J-1} \dfrac{324 \Lambda_2}{M^4 \eta^{1/2} \beta^2 N^6} \mathbb{E} \left[ \left\langle \sum_{i=1}^N \int_{k\eta}^{k\eta + \cdot} f'(\lambda_i^{N,\beta}(r)) dW_i(r) \right\rangle_{\eta}^2 \right] \nonumber \\
		\le& \sum_{k=0}^{J-1} \dfrac{324 \Lambda_2}{M^4 \eta^{1/2} \beta^2 N^6} \mathbb{E} \left[ \left( \sum_{i=1}^N \int_{k\eta}^{(k+1)\eta} \left| f'(\lambda_i^{N,\beta}(r)) \right|^2 dr \right)^2 \right] 
		\le \dfrac{324 \Lambda_2 J \eta^{3/2}}{M^4 \beta^2 N^4} \|f'\|_{L^{\infty}}^4
		\le \dfrac{324 \Lambda_2 T \eta^{1/2}}{M^4 \beta^2 N^4} \|f'\|_{L^{\infty}}^4.
	\end{align}
	Hence, noting that $\beta \in \{1,2\}$ and $\eta^{1/8} > \eta \|f'\|_{L^{\infty}}$, by \eqref{eq-1.15-Holder estimate} and \eqref{eq-1.16-martingale term}, for $M>0$, we have
	\begin{align} \label{eq-1.17-equicontinuous estmation}
		& \mathbb{P} \left( \sup_{|t-s|\le \eta} |\langle f, L_N^{\beta}(t) \rangle - \langle f, L_N^{\beta}(s) \rangle| \ge (M+1) \eta^{1/8} \right) \nonumber\\
	  \le~&  \mathbb{P} \left( \sup_{|t-s|\le \eta} \left| \dfrac{\sqrt{2}}{\sqrt{\beta N^3}} \sum_{i=1}^N \int_s^t f'(\lambda_i^{N,\beta}(r)) dW_i(r) \right| \ge (M+1) \eta^{1/8} - \dfrac{\eta \|f'\|_{L^{\infty}}}{2} \right) \nonumber \\
		\le~& \mathbb{P} \left( \sup_{|t-s|\le \eta} \left| \dfrac{\sqrt{2}}{\sqrt{\beta N^3}} \sum_{i=1}^N \int_s^t f'(\lambda_i^{N,\beta}(r)) dW_i(r) \right| \ge M \eta^{1/8} \right) 
		\le \dfrac{324 \Lambda_2 T \eta^{1/2}}{M^4 N^4} \|f'\|_{L^{\infty}}^4.
	\end{align}
	
	Let $\{\tilde{f}_k\}_{k\in \bN}$ be a family of $C_b^2(\bR)$ functions that is dense in $C_0(\bR)$. Choose $\varepsilon_k = \left( 1 + k \| \tilde{f}_k' \|_{L^{\infty}} \right)^{-1}$ and define
	\begin{align*}
		C_T(\tilde{f}_k, \varepsilon_k)
		&= \bigcap_{n=1}^{\infty} \left\{ \mu \in C([0,T], \cP(\mathbb{R})): \sup_{|t-s|\le n^{-4}} \left| \langle \tilde{f}_k, \mu_t \rangle - \langle \tilde{f}_k, \mu_s \rangle \right| \le \dfrac{1}{\varepsilon_k \sqrt{n}} \right\} \\
		&= \left\{ \mu \in C([0,T], \cP(\mathbb{R})): \sup_{|t-s|\le n^{-4}} \left| \langle \tilde{f}_k, \mu_t \rangle - \langle \tilde{f}_k, \mu_s \rangle \right| \le \dfrac{1}{\varepsilon_k \sqrt{n}}, \forall n \in \mathbb{N} \right\} \\
		&= \left\{ \mu \in C([0,T], \cP(\mathbb{R})): t \rightarrow \langle \tilde{f}_k, \mu_t \rangle \in C_{\left\| \tilde{f}_k \right\|_{L^{\infty}}}(\{(\varepsilon_k \sqrt{n})^{-1}\}, \{n^{-4}\}) \right\},
	\end{align*}
	where the set
    \[
	C_{M}(\{(\varepsilon_k \sqrt{n})^{-1}\}, \{n^{-4}\})
	= \bigcap_{n=1}^{\infty} \left\{ g \in C([0,T], \mathbb{R}): \sup_{|t-s|\le n^{-4}} |g(t) - g(s)| \le (\varepsilon_k \sqrt{n})^{-1}, \sup_{t\in[0,T]} |g(t)| \le M \right\},
    \]
	is (sequentially) compact in $C([0,T],\mathbb R)$ according to Arzela-Ascoli Lemma. By \eqref{eq-1.17-equicontinuous estmation},	
	\begin{align} \label{eq-1.18-compact 1}
		&\sum_{N=1}^{\infty} \sum_{k \ge 1} \mathbb{P} (L_N^{\beta} \notin C_T(\tilde{f}_k, \varepsilon_k)) 
		\le \sum_{N=1}^{\infty} \sum_{k \ge 1} \sum_{n=1}^{\infty} \mathbb{P} \left( \sup_{|t-s|\le n^{-4}} \left| \langle \tilde{f}_k, L_N^{\beta}(t) \rangle - \langle \tilde{f}_k, L_N^{\beta}(s) \rangle \right| > \dfrac{1}{\varepsilon_k \sqrt{n}} \right) \nonumber \\
		\le~& \sum_{N=1}^{\infty} \sum_{k \ge 1} \sum_{n=1}^{\infty} \dfrac{324 \Lambda_2 T}{(\varepsilon_k^{-1}-1)^4 n^2 N^4} \|\tilde{f}_k'\|_{L^{\infty}}^4 
		= 324 \Lambda_2 T \sum_{n=1}^{\infty} n^{-2} \sum_{k \ge 1} \dfrac{\| \tilde{f}_k' \|_{L^{\infty}}^4}{(\varepsilon_k^{-1}-1)^4} \sum_{N=1}^{\infty} \dfrac{1}{N^4} \nonumber \\
		=~ & 324 \Lambda_2 T \sum_{n=1}^{\infty} n^{-2} \sum_{k \ge 1} k^{-4} \sum_{N=1}^{\infty} \dfrac{1}{N^4}
		<\infty,
	\end{align}
	Since the function $\varphi$ is positive and tends to infinity as $|x| \rightarrow +\infty$, the set
	\begin{align*}
		K(\varphi) =	\left\{ \mu \in\cP(\mathbb{R}): \langle \varphi, \mu \rangle \le 1 + C_0 + T \left\| \varphi'' \right\|_{L^{\infty}} \right\}
	\end{align*}
	is tight, i.e., it is (sequentially) compact in $\cP(\mathbb{R})$. By \eqref{eq-1.15-Holder estimate} for $f=\varphi$ and $s=0$, Markov inequality and Burkholder-Davis-Gundy inequality, we have
	\begin{align} \label{eq-1.19-compact 2}
		& \sum_{N=1}^{\infty} \mathbb{P} \left( \exists t \in [0,T], \ \mathrm{s.t.} \ L_N^{\beta}(t) \notin K(\varphi) \right) 
		= \sum_{N=1}^{\infty} \mathbb{P} \left( \sup_{t\in[0,T]} \langle \varphi, L_N(t) \rangle > 1 + C_0 + T \left\| \varphi'' \right\|_{L^{\infty}} \right) \notag \\
		\le~& \sum_{N=1}^{\infty} \mathbb{P} \left( \sup_{t\in[0,T]} \left| \dfrac{\sqrt{2}}{\sqrt{\beta N^3}} \sum_{i=1}^N \int_0^t \varphi'(\lambda_i^{N,\beta}(r)) dW_i(r) \right| > 1 \right) 
		\le \sum_{N=1}^{\infty} \bE \left[ \sup_{t\in[0,T]} \left| \dfrac{\sqrt{2}}{\sqrt{\beta N^3}} \sum_{i=1}^N \int_0^t \varphi'(\lambda_i^{N,\beta}(r)) dW_i(r) \right|^2 \right] \nonumber \\
		\le~& \sum_{N=1}^{\infty} \Lambda_1 \bE \left[ \left\langle \dfrac{\sqrt{2}}{\sqrt{\beta N^3}} \sum_{i=1}^N \int_0^{\cdot} \varphi'(\lambda_i^{N,\beta}(r)) dW_i(r) \right\rangle_T \right] 
		= \sum_{N=1}^{\infty} \dfrac{2\Lambda_1}{\beta N^3} \bE \left[ \sum_{i=1}^N \int_0^T \left( \varphi'(\lambda_i^{N,\beta}(r)) \right)^2 dr \right] \nonumber \\
		\le~& \sum_{N=1}^{\infty} \dfrac{2 \Lambda_1 T \left\| \varphi' \right\|_{L^{\infty}}^2}{\beta N^2}
		< \infty.
	\end{align}
	
	By \cite[Lemma 4.3.13]{Anderson2010} (see also \cite[Lemma B.4]{Song2020}), the set
	\begin{align*}
		\mathcal{H} = \Bigl\{\mu\in C([0,T], \cP(\mathbb{R})): \mu_t \in K(\varphi), \ \forall t \in [0,T]\Bigr\} \cap \bigcap_{k\ge 1} C_T(\tilde{f}_k, \varepsilon_k)
	\end{align*}
	is compact in $C([0,T], \cP(\mathbb{R}))$. By \eqref{eq-1.18-compact 1} and \eqref{eq-1.19-compact 2}, we have
\[
		\sum_{N=1}^{\infty} \mathbb{P} (L_N^{\beta} \notin \mathcal{H}) 
		\le \sum_{N=1}^{\infty} \mathbb{P} (\exists t \in [0,T], \ \mathrm{s.t.} \ L_N^{\beta}(t) \notin K(\varphi))
		+ \sum_{N=1}^{\infty} \sum_{k \ge 1} \mathbb{P} (L_N^{\beta} \notin C_T(\tilde{f}_k, \varepsilon_k)) 
		< \infty.
\]
	Therefore, the Borel-Cantelli Lemma implies that 
    $\displaystyle
	  \mathbb{P} \left( \liminf_{N \rightarrow \infty} \{L_N^{\beta} \in \mathcal{H} \} \right) = 1$.

	Finally, the relative compactness of the family $\{L_N^{\beta}\}_{N \in \bN}$ follows from the compactness of $\mathcal{H}$.
	
	\vspace*{1\baselineskip}
	
	\noindent {Step 3:} We derive the equation \eqref{eq-equation for Dyson limit measure} for any limit point $\mu$ of the sequence $\{L_N^{\beta}(t), t\in[0,T]\}_{N \in \bN}$.
	
	Let $\{N_k\}_{k \in \bN}$ be a subsequence such that $L_{N_k}^{\beta}$ converges to $\mu$ in $C([0,T], \cP(\mathbb{R}))$. For any $\epsilon>0$, for any $f \in C_b^2(\bR)$, by Markov inequality and Burkholder-Davis-Gundy inequality, we have
	\begin{align*}
		& \sum_{k=1}^{\infty} \mathbb{P} \left( \sup_{t\in[0,T]} \left| \dfrac{\sqrt{2}}{\sqrt{\beta N_k^3}} \sum_{i=1}^{N_k} \int_0^t f'(\lambda_i^{N_k,\beta}(r)) dW_i(r) \right| \ge \epsilon \right) 
		\le \sum_{k=1}^{\infty} \epsilon^{-2} \bE \left[ \sup_{t\in[0,T]} \left| \dfrac{\sqrt{2}}{\sqrt{\beta N_k^3}} \sum_{i=1}^{N_k} \int_0^t f'(\lambda_i^{N_k,\beta}(r)) dW_i(r) \right|^2 \right] \nonumber \\
		\le~& \sum_{k=1}^{\infty} \dfrac{\Lambda_1}{\epsilon^2} \bE \left[ \left\langle \dfrac{\sqrt{2}}{\sqrt{\beta N_k^3}} \sum_{i=1}^{N_k} \int_0^{\cdot} f'(\lambda_i^{N_k,\beta}(r)) dW_i(r) \right\rangle_T \right] 
		= \sum_{k=1}^{\infty} \dfrac{2\Lambda_1}{\epsilon^2\beta N_k^3} \bE \left[ \sum_{i=1}^{N_k} \int_0^T \left( f'(\lambda_i^{N_k,\beta}(r)) \right)^2 dr \right] 
		\le \sum_{k=1}^{\infty} \dfrac{2\Lambda_1 T}{\epsilon^2\beta N_k^2} \left\|f'\right\|_{L^{\infty}}^2,
	\end{align*}
	which is finite since $N_k \ge k$. By Borel–Cantelli Lemma,
	\begin{align} \label{eq-1.20-martingale->0}
		\dfrac{\sqrt{2}}{\sqrt{\beta N_k^3}} \sum_{i=1}^{N_k} \int_0^t f'(\lambda_i^{N_k,\beta}(r)) dW_i(r)
		\to 0, \quad k \to \infty,
	\end{align}
	uniformly with respect to $t$ almost surely. Moreover, the boundedness of $\|f''\|_{L^{\infty}}$ yields
	\begin{align} \label{eq-1.21-drift->0}
		\left( \dfrac{1}{\beta} - \dfrac{1}{2} \right) \dfrac{1}{N_k} \int_0^t \langle f'', L_{N_k}^{\beta}(s) \rangle ds,
		\to 0, \quad k \to \infty,
	\end{align}
	uniformly with respect to $t$ almost surely. Therefore, by considering the subsequence $\{N_k\}$ in \eqref{eq-Ito formula}, and using \eqref{eq-1.20-martingale->0} and \eqref{eq-1.21-drift->0}, we have
	\begin{align*}
		\langle f, \mu_t \rangle
		= \langle f, \mu_0 \rangle
		+ \dfrac{1}{2} \int_{0}^t \iint_{\bR^2} \dfrac{f'(x) - f'(y)}{x-y} \mu_s(dx) \mu_s(dy) ds.
	\end{align*}
	
	\vspace*{1\baselineskip}
	
	\noindent {Step 4:} We establish the uniqueness of the solution to the equation \eqref{eq-equation for Dyson limit measure}. For simplicity, we only prove the uniqueness for the self-similar solution under null initial condition $X^{N,\beta}(0) = 0$. The idea can be found in \cite{Song2019} and \cite[Exercise 4.3.18]{Anderson2010}. For general case, we refer the interested readers to \cite[Lemma 4.3.15]{Anderson2010}. Note that under the null initial condition, the limit points of the sequence $\{L_N^{\beta}(t), t\in[0,T]\}_{N \in \bN}$ inherit the self-similarity property from the Brownian motions on the matrix entries, and hence the uniqueness of the limit of $\{L_N^{\beta}(t), t\in[0,T]\}_{N \in \bN}$.
	
	To prove the uniqueness, it is convenient to choose $f(x) = (z-x)^{-1}$ for $z \in \bC \setminus \bR$. Denote
	\begin{align*}
		G_t(z) = \int \dfrac{1}{z-x} \mu_t(dx),
	\end{align*}
	which is known as the Stieltjes transform of the measure $\mu_t$. We refer to \cite[Section 2.4.3]{Tao2012} and \cite[Section 1.3.2]{Bai2010} for more details about Stieltjes transform.
	
	Recall the matrix Brownian motion $H^{N,\beta}(t)$ in Definition \ref{Def-Dyson}, by the self-similarity of Brownian motion, under null initial condition, we have $X^{N,\beta}(t) \overset{d}{=} \sqrt{t} X^{N,\beta}(1)$, where $\overset{d}{=}$ is the equality in distribution. Thus, we have the scaling property of the Stieltjes transform
	\begin{align} \label{eq-1.23-scaling of G}
		G_t(z) = \dfrac{1}{\sqrt{t}} G_1 \left( \dfrac{z}{\sqrt{t}} \right).
	\end{align}
	Hence,
	\begin{align*}
		G_t(z) \partial_z G_t(z)
		= \dfrac{1}{\sqrt{t}^3} G_1 \left( \dfrac{z}{\sqrt{t}} \right) G_1' \left( \dfrac{z}{\sqrt{t}} \right)
		= - \dfrac{1}{z} \dfrac{d}{dt} \left( G_1^2 \left( \dfrac{z}{\sqrt{t}} \right) \right).
	\end{align*}
	Letting $f(X) = (z-x)^{-1}$ for $z \in \bC \setminus \bR$, \eqref{eq-equation for Dyson limit measure} can be written as
	\begin{align}
		G_t(z) =& G_0(z) - \int_{0}^t G_s(z) \partial_z G_s(z) ds \label{eq-1.23-integral equation for G}
		= G_0(z) + \dfrac{1}{z} G_1^2 \left( \dfrac{z}{\sqrt{s}} \right) \Bigg|_{s=0}^{s=t} 
		= G_0(z) + \dfrac{1}{z} G_1^2 \left( \dfrac{z}{\sqrt{t}} \right),
	\end{align}
	where we use
	\begin{align*}
		G_1 \left( \dfrac{z}{\sqrt{s}} \right) \Bigg|_{s=0}
		= \int \dfrac{1}{z/\sqrt{s}-x} \mu_t(dx) \Bigg|_{s=0}
		= \int \dfrac{\sqrt{s}}{z-\sqrt{s}x} \mu_t(dx) \Bigg|_{s=0}
		= 0.
	\end{align*}
	Letting $t=1$ in \eqref{eq-1.23-integral equation for G} and noting that $G_0(z) = 1/z$, we have
$		G_1^2(z) - z G_1(z) + 1 = 0,$
	of which the solution is
	\begin{align} \label{eq-1.25-solution of G_1}
		G_1(z) = \dfrac{z - \sqrt{z^2 - 4}}{2}.
	\end{align}
	Note that by definition,
	\begin{align*}
		\Im \left( G_t(z) \right)
		= \Im \left( \int \dfrac{1}{z-x} \mu_t(dx) \right)
		= \Im \left( \int \dfrac{\bar{z}-x}{|z-x|^2} \mu_t(dx) \right)
		= - \Im (z) \int \dfrac{1}{|z-x|^2} \mu_t(dx).
	\end{align*}
	Here, we use the notation $\Im(w)$ for the imaginary part of $w$. Thus, for all $t$, $G_t(z)$ maps $z \in \bC_+$ to $\bC_-$. Thus, the square root in \eqref{eq-1.25-solution of G_1} should be the branch that maps from $\mathbb{C}_{+}$ to $\mathbb{C}_{+}$.
	
	Lastly, \eqref{eq-1.23-scaling of G} and \eqref{eq-1.25-solution of G_1} yield
	\begin{align} \label{eq-2.38}
		G_t(z) = \dfrac{z - \sqrt{z^2-4t}}{2t},
	\end{align}
	which is the unique self-similar solution to \eqref{eq-equation for Dyson limit measure}.
\end{proof}

\begin{remark}
	The equation \eqref{eq-2.38} is consistent with the Stieltjes transform of the semi-circle law \cite[(2.4.6)]{Anderson2010}, \cite[(2.103)]{Tao2012} and \cite[Lemma 2.11]{Bai2010}.
\end{remark}

\begin{remark}
	Note that under the null initial condition $X^{N,\beta}(0)=0$, at the time $t=1$, the matrix $X^{N,\beta}(1) = H^{N,\beta}(1)$ is the Gaussian Orthogonal Ensemble (GOE) when $\beta=1$, and the Gaussian Unitary Ensemble (GUE) when $\beta=2$. We refer to \cite[Section 2.3]{Tao2012} for more details about GOE and GUE. Moreover, \eqref{eq-1.25-solution of G_1} is the Stieltjes transform of the famous semi-circle law (see \cite[(2.4.6)]{Anderson2010}, \cite[(2.103)]{Tao2012} and \cite[Lemma 2.11]{Bai2010}). Thus, Theorem \ref{Thm-Wigner dynamic-strong} gives a dynamical proof of the semi-circle law (see \cite[Theorem 2.4.2]{Tao2012} or \cite[Theorem 2.5]{Bai2010} for the statement of semi-circle law).
\end{remark}

\begin{remark}
	The differential form of \eqref{eq-1.23-integral equation for G},
	\begin{align*}
		\partial_t G_t(z) + G_s(z) \partial_z G_s(z) = 0,
	\end{align*}
	is the complex version of inviscid Burgers' equation.
\end{remark}

In some literature (see for example \cite{Rogers1993, Jaramillo2019, Pardo2017, Pardo2016, Song2020}), some other tightness argument was used to obtain the convergence in law of the eigenvalue empirical measure processes in $C([0,T], \cP(\bR))$. To illustrate this argument, we present the following weak version of Theorem \ref{Thm-Wigner dynamic-strong}.

\begin{theorem}[\cite{pt07}, Proposition 3.1] \label{Thm-Wigner dynamic-weak}
	Assume that all the conditions in Theorem \ref{Thm-Wigner dynamic-strong} hold. Then the sequence $\{L_N^{\beta}(t), t\in[0,T]\}_{N \in \bN}$ converges {in probability} in $C([0,T], \cP(\bR))$. Moreover, its limit $\mu$ is characterized by the equation \eqref{eq-equation for Dyson limit measure}.
\end{theorem}

\begin{proof}[\textbf{\upshape Proof:}]
	The equation \eqref{eq-equation for Dyson limit measure} can be derived as in Theorem \ref{Thm-Wigner dynamic-strong} by using It\^{o} calculus and martingale theory. We only prove the convergence in law of the sequence $\{L_N^{\beta}(t), t\in[0,T]\}_{N \in \bN}$, noting that the limit is a deterministic measure. The key idea is to obtain the following moment estimation
	\begin{align} \label{eq-tightness moment}
		\bE \left[ \left| \left\langle f, L_N^{\beta}(t) \right\rangle - \left\langle f, L_N^{\beta}(s) \right\rangle \right|^{1+a} \right]
		\le C_{f,T} |t-s|^{1+b}, ~ \forall t,s \in [0,T], ~ \forall N \in \bN,
	\end{align}
	for some positive constants $a$ and $b$, and for all $f \in C^2(\bR)$ with bounded first and second derivatives. Here $C_{f,T}$ is a positive constant depending only on $f$ and $T$. Then the tightness of the sequence $\{L_N^{\beta}(t), t\in[0,T]\}_{N \in \bN}$ in $C([0,T], \cP(\bR))$ follows from \eqref{eq-tightness moment} and \cite[Proposition B.3]{Song2020}.
	
	To establish \eqref{eq-tightness moment}, one may first obtain \eqref{eq-1.15-Holder estimate} by It\^{o} calculus and then apply the Burkholder-Davis-Gundy inequality to get an upper bound for some even moment of the martingale term. In the following, we provide another approach to get \eqref{eq-tightness moment}, where the pathwise H\"{o}lder continuity of the matrix entries is used. This idea can also be found in \cite{Jaramillo2019, Pardo2016, Pardo2017, Song2020}.
	
	It is well known that almost all the paths of Brownian motion are
    $(1/2-\varepsilon)$-H\"{o}lder continuous for any
    $\varepsilon\in(0,1/2)$, and so are the paths of the entries in
    $X^{N,\beta}$. Consider the H\"{o}lder norm of the matrix entries
    $X_{i,j}^{N,\beta}(t)$,  for $1 \le i \le j \le N$,
	\begin{align*}
		\left\|X_{i,j}^{N,\beta} \right\|_{0,T; 1/2-\varepsilon}
		= \sup_{0 \le s < t \le T} \dfrac{\left| X_{i,j}^{N,\beta}(t) - X_{i,j}^{N,\beta}(s) \right|}{|t-s|^{1/2-\varepsilon}}.
	\end{align*}
	By the Fernique Theorem (\cite{Fernique1970}), we have the following estimation
	\begin{align} \label{eq-1.28-BM pathwise holder}
		\bE \left[ \exp \left(\alpha N \left\|X_{i,j}^{N,\beta} \right\|_{0,T; 1/2-\varepsilon}^2 \right) \right] < \infty,
	\end{align}
	where $\alpha=\alpha(\varepsilon, T)$ is a positive constant depending on $(\varepsilon, T)$ and $X$.
	
	By mean value theorem and Hoffman-Wielandt inequality (\cite[Lemma 2.1.19]{Anderson2010}), we have
	\begin{align} \label{eq-1.29-pathwise holder}
		& \left| \left\langle f, L_N^{\beta}(t) \right\rangle - \left\langle f, L_N^{\beta}(s) \right\rangle \right|^2
		= \left| \dfrac{1}{N} \sum_{i=1}^N \left( f \left( \lambda_i^{N,\beta} (t) \right) - f \left( \lambda_i^{N,\beta} (s) \right) \right) \right|^2 \nonumber \\
		\le~& \dfrac{\|f'\|_{L^{\infty}}^2}{N} \sum_{i=1}^N \left|  \lambda_i^{N,\beta} (t) - \lambda_i^{N,\beta} (s) \right|^2 
		\le \dfrac{\|f'\|_{L^{\infty}}^2}{N} \sum_{i,j=1}^N \left| X^{N,\beta}(t) - X^{N,\beta}(s) \right|^2.
	\end{align}
	This together with Minkowski inequality, the pathwise H\"{o}lder continuity of $X_{i,j}^{N,\beta}$ and Cauchy-Schwarz inequality yields
	\begin{align} \label{eq-1.30-pathwise holder}
		&\bE \left[ \left| \left\langle f, L_N^{\beta}(t) \right\rangle - \left\langle f, L_N^{\beta}(s) \right\rangle \right|^4 \right]
		\le \dfrac{\|f'\|_{L^{\infty}}^4}{N^2} \bE \left[ \left( \sum_{i,j=1}^N \left| X^{N,\beta}(t) - X^{N,\beta}(s) \right|^2 \right)^2 \right] 
		\le \dfrac{\|f'\|_{L^{\infty}}^4}{N^2} \left( \sum_{i,j=1}^N \left( \bE \left[ \left| X^{N,\beta}(t) - X^{N,\beta}(s) \right|^4 \right] \right)^{1/2} \right)^2 \nonumber \\
		\le~& \dfrac{\|f'\|_{L^{\infty}}^4}{N^2} \left( \sum_{i,j=1}^N \left( \bE \left[ \left\|X_{i,j}^{N,\beta} \right\|_{0,T; 1/2-\varepsilon}^4 |t-s|^{2-4\varepsilon} \right] \right)^{1/2} \right)^2 
		= \dfrac{\|f'\|_{L^{\infty}}^4 |t-s|^{2-4\varepsilon}}{N^4} \left( \sum_{i,j=1}^N \left( \bE \left[ N^2 \left\|X_{i,j}^{N,\beta} \right\|_{0,T; 1/2-\varepsilon}^4 \right] \right)^{1/2} \right)^2 \nonumber \\
		\le& \dfrac{\|f'\|_{L^{\infty}}^4 |t-s|^{2-4\varepsilon}}{N^2} \sum_{i,j=1}^N \bE \left[ N^2 \left\|X_{i,j}^{N,\beta} \right\|_{0,T; 1/2-\varepsilon}^4 \right].
	\end{align}
	Recall that the matrix $X^{N,\beta}$ is symmetric for $\beta=1$ and Hermitian for $\beta=2$. Also note that the upper-diagonal entries are i.i.d., as well as the diagonal entries. Thus,
	\begin{align} \label{eq-1.31-Holder norm-matrix entries}
		 \dfrac{1}{N^2} \sum_{i,j=1}^N \bE \left[ N^2 \left\|X_{i,j}^{N,\beta} \right\|_{0,T; 1/2-\varepsilon}^4 \right] 
		\le~ & \bE \left[ N^2 \left\|X_{1,1}^{N,\beta} \right\|_{0,T; 1/2-\varepsilon}^4 \right] + \bE \left[ N^2 \left\|X_{1,2}^{N,\beta} \right\|_{0,T; 1/2-\varepsilon}^4 \right] \nonumber \\
		\le~ & \dfrac{2}{\alpha^2} \bE \left[ \exp \left(\alpha N \left\|X_{1,1}^{N,\beta} \right\|_{0,T; 1/2-\varepsilon}^2 \right) \right]
		+ \dfrac{2}{\alpha^2} \bE \left[ \exp \left(\alpha N \left\|X_{1,2}^{N,\beta} \right\|_{0,T; 1/2-\varepsilon}^2 \right) \right],
	\end{align}
	where we use the inequality $x^2/2 \le e^x$ for $x \ge 0$. Therefore, by \eqref{eq-1.30-pathwise holder}, \eqref{eq-1.31-Holder norm-matrix entries} and \eqref{eq-1.28-BM pathwise holder}, we obtain \eqref{eq-tightness moment} with $a=3$ and $b=1-4\varepsilon$ for $\varepsilon < 1/4$.
\end{proof}

\begin{remark}
	The weak semi-circle law, which can be found in \cite[Theorem 2.1.1]{Anderson2010}, is recovered when we choose $X^{N,\beta}(0)=0$ and $t=1$ in Theorem \ref{Thm-Wigner dynamic-weak}. Thus, Theorem \ref{Thm-Wigner dynamic-weak} gives a dynamical proof of the weak semi-circle law.
\end{remark}

In the framework of free probability theory, \cite{Voi91, VDN92, Biane1997} showed that independent $N \times N$ random matrices converge to free random variables as $N$ tends to infinity. In this sense, the large $N$ limit of Brownian motion with values in the space of $N \times N$ Hermitian matrices is known as {free Brownian motion} (\cite[Theorem 1]{Biane1997}). More precisely, a (one-side) free (additive) Brownian motion $\{S(t), t \ge 0 \}$ is a non-commutative stochastic process that satisfies:
\begin{itemize}
	\item $S(0) = 0$;
	\item For $t_2 > t_1 \ge 0$, the law of $S(t_2) - S(t_1)$ is the semicircular distribution with mean $0$ and variance $t_2-t_1$;
	\item For all $n \in \bN$, and $t_n > \cdots > t_1 \ge 0$, the increments $S(t_1), S(t_2) - S(t_1), \ldots, S(t_n) - S(t_{n-1})$ are freely independent.
\end{itemize}
We refer the interested reader to \cite{Nourdin2014, HP00, MS17} and the references therein for this topic.

For the complex model ($\beta=2$), for $p \in \bN$, the $p$-th moment of the sequence of eigenvalue empirical measure processes $\{L_N^2(t), t \in [0,T] \}_{N \in \bN}$ was considered in \cite{pt07}, the motivation of which came from the study of moments for the GUE in \cite{Mehta2004}. By using It\^{o} calculus and martingale theory, \cite{pt07} established a recursive formula for the sequence $\{\langle x^{p}, L_N^2(t) \rangle, t \ge 0 \}_{p \in \bN}$ for $N \in \bN$, and proved that for $p \in \bN$, the sequence $\{\langle x^{2p}, L_N^2(t) \rangle, t \ge 0 \}_{N \in \bN}$ converges to $\langle x^{2p}, \mu_t \rangle$ uniformly in $t \in [0,T]$ almost surely and in $L^{2q}$ with $q \ge 1$. Moreover, \cite{pt07} also investigated the largest and least eigenvalue processes and showed that
\begin{align*}
	\max_{t \in [0,T]} \lambda_1^{N,\beta}(t) \to 2\sqrt{T}, ~
	\min_{t \in [0,T]} \lambda_N^{N,\beta}(t) \to -2\sqrt{T}, ~
	\text{as} ~ N \to \infty,
\end{align*}
almost surely.

It is natural to consider the fluctuation of the sequence $\{L_N^{\beta}(t)\}_{N \in \bN}$ around its limit $\mu$. Consider the random fluctuations
\begin{align*}
	\cL_N^{\beta}(f)(t) = N \big( \langle f,L_N^{\beta}(t) \rangle - \langle f, \mu_t \rangle \big),
\end{align*}
for test function $f$ belonging to some proper function space. For the complex Dyson Brownian motion \eqref{eq-Dyson BM SDE}, \cite[Theorem 1.1]{cd01} established the central limit theorem (CLT) for Chebyshev polynomials with null initial condition. Note that for monomials $f(x) = x^p$, $\cL_N^{2}(x^p)(t)$ is the fluctuation of the $p$th moment processes $\langle x^p, L_N^{2}(t) \rangle$ around the $p$th moment of the corresponding limit measure. By martingale theory, \cite[Theorem 4.3]{pt07} proved the convergence in distribution of $\cL_N^{2}(x^p)(t)$ to a centred Gaussian process $\cL^{2}(x^p)(t)$ characterized by a recursive formula. The CLT for the sequence $\{L_N^{\beta}(t)\}_{N \in \bN}$ with polynomial test functions was obtained in \cite[Theorem 4.3.20]{Anderson2010} and is presented below.

\begin{theorem}[\cite{Anderson2010}, Theorem 4.3.20] \label{Thm-CLT for Dyson}
	Let $T>0$ be a fixed number. Assume
    $\displaystyle
	\sup_{N\in\bN} \max_{1 \le i \le N} \left| \lambda_i^{N,\beta}(0) \right| < \infty
    $,
	and for all $n \in \mathbb N$, $p\ge 1$,
	\begin{align*}
		\sup_{N\in\bN} \bE \left[ \left| N \left( \left\langle x^n, L_N^{\beta}(0) \right\rangle - \left\langle x^n, \mu_0 \right\rangle \right) \right|^p \right]<\infty.
	\end{align*}
	Furthermore, assume that for any $f(x) \in \mathbb{C}[x]$, the initial value $\mathcal{L}_N^{\beta}(f)(0)$ converges in probability to a random variable $\cL^{\beta}(f)(0)$. Here, $\mathbb{C}[x]$ is the set of polynomials with complex coefficients.
	
	Then there exists a family of processes $\{\cL^{\beta}(f)(t), t\in [0,T]\}_{f \in \mathbb{C}[x]}$, such that for any $n \in \mathbb{N}$ and any polynomials $P_1, \ldots, P_n \in \mathbb{C}[x]$, the vector-valued process $\{(\mathcal{L}_N^{\beta}(P_1)(t), \ldots, \mathcal{L}_N^{\beta}(P_n)(t)), t \in [0,T]\}_{N \in \bN}$ converges in distribution to $\{(\cL^{\beta}(P_1)(t), \ldots, \cL^{\beta}(P_n)(t)), t \in [0,T] \}$.
	
	The limit processes $\{\cL^{\beta}(f)(t), t\in [0,T]\}_{f \in \mathbb{C}[x]}$ are characterized by the following properties:
	\begin{enumerate}
		\item For $P_1, P_2 \in \mathbb{C}[x]$, $\alpha_1, \alpha_2 \in \mathbb{C}$, $t \in [0,T]$,
		\begin{align*}
			\cL^{\beta}(\alpha_1 P_1 + \alpha_2 P_2)(t) = \alpha_1 \cL^{\beta}(P_1)(t) + \alpha_2 \cL^{\beta}(P_2)(t).
		\end{align*}
		\item The basis $\{\cL^{\beta}(x^n)(t), t\in [0,T]\}_{n \in \mathbb{N}}$ of $\{\cL^{\beta}(f)(t), t\in [0,T]\}_{f \in \mathbb{C}[x]}$ is characterised by
		\begin{align*}
			\cL^{\beta}(1)(t) = 0, \quad \cL^{\beta}(x)(t) = \cL^{\beta}(x)(0) + G_t^{\beta}(x),
		\end{align*}
		and for $n \ge 0$,
		\begin{align*} 
			\cL^{\beta}(x^{n+2})(t)
			= &\cL^{\beta}(x^{n+2})(0) + \dfrac{2-\beta}{2\beta} (n+2)(n+1) \int_{0}^t \langle x^n, \mu_s \rangle ds
			 + (n+2) \sum_{k=0}^n \int_0^t \cL^{\beta}(x^{n-k})(s) \langle x^k, \mu_s \rangle ds + G_t^{\beta}(x^{n+2}),
		\end{align*}
		where $\{G_t^{\beta}(x^n), t \in [0,T]\}_{n \in \mathbb{N}}$ is a family of centred Gaussian processes with the covariance
		\begin{align*} 
			\mathbb{E} \left[ G_t^{\beta}(x^n) G_s^{\beta}(x^m) \right]
			= \dfrac{2 mn}{\beta} \int_0^{t \wedge s} \langle x^{n+m-2}, \mu_u \rangle du, \quad n,m \ge 1.
		\end{align*}
	\end{enumerate}
\end{theorem}

\section{Positive-definite symmetric matrix valued processes} \label{Sec:positive-definite matrices}

\subsection{Brownian motions of ellipsoids} \label{Sec:BM of ellipsoids}

The study of stochastic processes with values in the space of positive-definite symmetric matrices, or the space of { ellipsoids}, can be dated back to at least \cite{Dynkin1961}, where a class of Markov processes were studied by using differential geometry. Later, the Brownian motions of ellipsoids were considered in \cite{Norris1986}, and some of the results in \cite{Dynkin1961, Orihara1970} were recovered without using differential geometry.

Let $B(t)$ be a $N \times N$ matrix whose entries are i.i.d. standard Brownian motions (matrix Brownian motion). Let $G^N(t)$ be a process on the group of invertible $N \times N$ matrices that solves the following matrix SDE
\begin{align*}
	dG^N(t) = dB(t) \circ G^N(t).
\end{align*}
Then the process $\{G^N(t+u)G^N(u)^{-1}: t \ge 0\}$ is identical in law to the process $G^N(t)$ and is independent of the process $\{G^N(r):r \in [0,u]\}$ for all $u>0$. The process $G^N(t)$ is known as the right-invariant Brownian motion. Let $X^N(t) = G^N(t) G^N(t)^\intercal$ and $Y^N(t) = G^N(t)^\intercal G^N(t)$, which are both Markov processes on the space of ellipsoids. The process $Y^N(t)$ is known as Dynkin’s Brownian motion.

Suppose that $G^N(0)$ is chosen such that $X^N(0)$ has distinct eigenvalues.

\begin{theorem}[\cite{Norris1986}, Theorem A] \label{Thm-BM ellipsoids-SDE}
	The eigenvalue processes of $X^N(t)$ never collide and never hit $0$ for all $t>0$ almost surely. The ordered eigenvalue processes $\lambda_1^N(t) > \cdots > \lambda_N^N(t) (>0)$ satisfy the following system of SDEs
	\begin{align} \label{eq-eigenvalue SDE ellipsoids}
		\dfrac{1}{2} d\left( \ln \lambda_i^N(t) \right)
		= dW_i(t) + \dfrac{1}{2} \sum_{j:j\neq i} \dfrac{\lambda_i^N(t) + \lambda_j^N(t)}{\lambda_i^N(t) - \lambda_j^N(t)} dt,
	\end{align}
	where $\{W_1(t), \ldots, W_N(t)\}$ are independent standard Brownian motions. Moreover,
	\begin{align*}
		\lim_{t\to \infty} \dfrac{\ln \lambda_i^N(t)}{t} = N+1-2i.
	\end{align*}
\end{theorem}

The system of SDEs for eigenvalue processes was derived in \cite{Norris1986} by using It\^{o} calculus and martingale theory as in Theorem \ref{Thm-Dyson SDE}. The almost sure non-collision of the eigenvalue processes was proved by the theorem of time-change for local martingales (see \cite[Chapter 3, Theorem 4.6]{Karatzas1998}), which is of the same spirit as the McKean's argument (Lemma \ref{Lemma-McKean argument}). The long time behavior of the eigenvalue processes was studied by constructing auxiliary processes with a comparison result.

The eigenvector processes of $X^N(t)$ and $Y^N(t)$ were also investigated in \cite{Norris1986} by It\^{o} calculus and their behavior is very different. The eigenvector matrix of $X^N(t)$ ultimately behaves like Brownian motion on $O(N)$, while that of $Y^N(t)$ converges to a limiting value.

\subsection{Wishart processes} \label{Sec:Wishart}

Wishart process was introduced in \cite{Bru1989} to perform principal component analysis on a set of resistance data of Escherichia Coli to certain antibiotics. Let $B(t)$ be a $N \times p$ matrix whose entries are i.i.d. standard real Brownian motions (matrix Brownian motion). The $N \times N$ symmetric matrix $X^N(t) = (B(t) + A) (B(t) + A)^\intercal$, where $A$ is a $N \times p$ real deterministic matrix, is the Wishart process. By \cite{Bru91}, the Wishart process $X^N(t)$ solves the following matrix SDE
\begin{align} \label{eq-Wishart matrix SDE}
	dX^N(t) = \sqrt{X^N(t)} dW(t) + dW(t)^{\intercal} \sqrt{X^N(t)} + pI_Ndt,
\end{align}
where $W(t)$ is a $N \times N$ matrix Brownian motion. The ordered eigenvalue processes $\lambda_1^N(t) \ge \lambda_2^N(t) \ge \cdots \ge \lambda_N^N(t)$ of $X^N(t)$ was studied first in \cite{Bru1989}.

\begin{theorem}[\cite{Bru1989}, Theorem 1] \label{Thm-Wishart SDE}
	Assume that $X^N(0)$ has $N$ distinct eigenvalues $\lambda_1^N(0) > \lambda_2^N(0) > \cdots > \lambda_N^N(0)$. Denote the first collision time of the eigenvalue processes by
	\begin{align*}
		\tau_N = \inf \left\{ t>0: \exists \ i \neq j, ~\lambda_i^N(t) = \lambda_j^N(t) \right\}.
	\end{align*}
	Then
	\begin{align*}
		\bP \left( \tau_N = + \infty \right) = 1.
	\end{align*}
	Furthermore, the ordered eigenvalue processes $\lambda_1^N(t), \cdots, \lambda_N^N(t)$ of $X^N(t)$ satisfy the following system of SDEs
	\begin{align} \label{eq-SDE-Wishart eigenvalue}
		d\lambda_i^N(t)
		= 2\sqrt{\lambda_i^N(t)} dW_i(t)
		+ \left( p + \sum_{j:j\neq i} \dfrac{\lambda_i^N(t) + \lambda_j^N(t)}{\lambda_i^N(t) - \lambda_j^N(t)} \right) dt, ~ i \in\{1, \ldots, N\},
	\end{align}
	where $\{W_1(t), \ldots, W_N(t)\}$ are independent standard Brownian motions.
\end{theorem}

Theorem \ref{Thm-Wishart SDE} can be proved following the idea of the proof of Theorem \ref{Thm-Dyson SDE}. Similarly, the system of SDEs for eigenvalue processes can be derived by It\^{o} calculus and martingale theory. The almost sure non-collision of the eigenvalue processes can also be proved by the McKean's argument.

The eigenvector processes were also studied in \cite{Bru1989} by It\^{o} calculus. Under the same assumption as in Theorem \ref{Thm-Wishart SDE}, with an appropriate choice of unit eigenvalue vector processes, the system of SDEs for them were established in \cite[Theorem 2]{Bru1989}.

The assumption that the eigenvalues of $X^N(0)$ are distinct in Theorem \ref{Thm-Wishart SDE} automatically implies that $p \ge N-1$. For the case $p < N$, by \cite[Corollary 1]{Bru91}, Theorem \ref{Thm-Wishart SDE} is still valid for the set of non-trivial eigenvalue processes $\lambda_1^N(t), \cdots, \lambda_p^N(t)$. Note that the Wishart processes $X^N(t)$ is positive semi-definite, and is degenerate when $p<N$ for all $t$. In some situations, it interesting to know whether the Wishart processes is non-degenerated, which is equivalent to know the infiniteness of the hitting time of the least eigenvalue processes on $0$. By using the McKean's argument, \cite[Proposition 1]{Bru91} proved that $\lambda_N^N(t) > 0$ for all $t$ almost surely for the case $p > N$. For the critical case, $p=N$, the set of hitting time on $0$ ($\{t: \lambda_N^N(t) = 0\}$) has zero Lebesgue measure almost surely. \cite{Bru91} also considered the matrix model \eqref{eq-Wishart matrix SDE} whenever $p>0$ is not an integer, and proved that the conclusion of Theorem \ref{Thm-Wishart SDE} holds for the unique solution to \eqref{eq-Wishart matrix SDE} for $p>N-1$ (\cite[Theorem 2]{Bru91}).

Let $Y^N(t)$ be the complex analogue of $X^N(t)$, that is, $Y^N(t) = (\tilde{B}(t) + \tilde{A}) (\tilde{B}(t) + \tilde{A})^\intercal$, where $\tilde{B}(t)$ is a $N \times p$ complex matrix whose real and imaginary parts are independent matrix Brownian motions, and $\tilde{A}$ is a $N \times p$ complex deterministic matrix. Then $Y^N(t)$ is known as Laguerre process (\cite{Konig2001}). With minor modification to the Wishart case (Theorem \ref{Thm-Wishart SDE}), the non-collision property of the eigenvalue processes can be established and the following system of SDEs for eigenvalue processes can be obtained
\begin{align} \label{eq-SDE-Laguerre eigenvalue}
	d\lambda_i^N(t)
	= 2\sqrt{\lambda_i^N(t)} dW_i(t)
	+ 2 \left( p + \sum_{j:j\neq i} \dfrac{\lambda_i^N(t) + \lambda_j^N(t)}{\lambda_i^N(t) - \lambda_j^N(t)} \right) dt, ~ i \in\{ 1, \ldots, N\}.
\end{align}
See for example \cite[(2)]{pt09}, \cite[(1.2)]{Katori2011}.

The eigenvalue processes \eqref{eq-SDE-Laguerre eigenvalue} were treated as particle system in \cite{Konig2001}, and they were proved to evolve like $N$ independent squared Bessel processes of dimension $2(p-N+1)$ conditioned to no collision among each other, assuming $p\ge N$. For more properties of particle systems related to Brownian motions, we refer to \cite{Katori2004, Katori2007, Katori2011}.

The high-dimensional limits of the normalized eigenvalue processes $\{\lambda_i^N(t)/N\}_{1 \le i \le N}$ of \eqref{eq-SDE-Laguerre eigenvalue} was studied in \cite{Duvillard2001} by proving large deviation bounds. Denote the empirical measure process by
\begin{align*}
	L_N(t)(dx) = \dfrac{1}{N} \sum_{i=1}^N \delta_{\lambda_i^N(t)/N} (dx).
\end{align*}
The almost sure weak convergence of the sequence $\{L_N(t)\}_{N\in\bN}$ as well as the differential equation satisfied by the limiting measure-valued processes was established in \cite[Corollary 3.1]{Duvillard2001}, assuming that $p/N$ converges to a positive number $c$. Moreover, the limit $\mu$ is the well-known Mar\v{c}enko-Pastur law (free Poisson distribution). The results were recovered in \cite[Theorem 3.3]{pt09}.

For $p \in \bN$, the $p$-th moment of the sequence $\{L_N(t), t \in [0,T] \}$ of normalized Laguerre process was considered in \cite{pt09}. By using It\^{o} calculus and martingale theory, \cite{pt09} established a recursive formula for the sequence $\{\langle x^p, L_N(t) \rangle: t \ge 0 \}_{p \in \bN}$ for $N \in \bN$, and proved that as $N \to \infty$, $\langle x^p, L_N(t) \rangle$ converges to $\langle x^p, \mu_t \rangle$ uniformly in $t \in [0,T]$ almost surely and in $L^q$ with $q \ge 1$ for $p \in \bN$. Moreover, \cite{pt07} also investigated the largest eigenvalue processes $\lambda_1^N(t)$ and the least eigenvalue processes $\lambda_N^N(t)$, and showed that
\begin{align*}
	\max_{t \in [0,T]} \lambda_1^N(t) \to (1+\sqrt{c})^2 \sqrt{T}, ~
	\min_{t \in [0,T]} \lambda_N^N(t) \to (1-\sqrt{c})^2 \sqrt{T}, ~
	\text{as} ~ N \to \infty,
\end{align*}
almost surely, where $c = \lim_{N \to \infty} p/N$.

The fluctuation of the sequence $\{L_N(t), t \in [0,T] \}$ around its limit $\mu$ has also been studied. Denote the random fluctuation
\begin{align*}
	\cL_N(f)(t) = N \left( \left\langle f, L_N(t) \right\rangle - \left\langle f, \mu_t \right\rangle \right)
\end{align*}
for an appropriate test function $f$. \cite[Theorem 2.5]{cd01} established the CLT for a class of polynomial functions with null initial condition. Note that for monomials $f(x) = x^p$, the random fluctuation $\cL_N(x^p)(t)$ is the fluctuation of the moment processes $\langle x^p, L_N(t) \rangle$ around the corresponding moment of the limit measure. By martingale theory, \cite[Theorem 4.3]{pt07} proved the convergence in distribution of $\cL_N(x^p)(t)$ to a centred Gaussian process $\cL(x^p)(t)$ whose distribution is characterized by recursive formulas.

The Wishart process \eqref{eq-Wishart matrix SDE} was generalized in \cite{Graczyk2019} to a symmetric matrix valued process that solves the following matrix SDE
\begin{align} \label{eq-matrix SDE general Wishart}
	dX^N(t) = \sqrt{\left| X^N(t) \right|} dW(t) + dW(t)^{\intercal} \sqrt{\left| X^N(t) \right|} + pI_Ndt.
\end{align}
Its ordered eigenvalue processes $\lambda_1^N(t) \ge \cdots \ge \lambda_N^N(t)$ satisfy the following system of SDEs
\begin{align} \label{eq-SDE Graczyk general Wishart eigenvalue}
	d\lambda_i^N(t)
	= 2\sqrt{\left| \lambda_i^N(t) \right|} dW_i(t)
	+ \left( p + \sum_{j:j\neq i} \dfrac{\left| \lambda_i^N(t) \right| + \left| \lambda_j^N(t) \right|}{\lambda_i^N(t) - \lambda_j^N(t)} \right) dt, ~ i \in\{ 1, \ldots, N\},
\end{align}
which is known as squared Bessel particle system. \cite[Theorem 1]{Graczyk2019} proved the existence and uniqueness of the non-colliding strong solution for all $p \in \bR$. The conditions for the uniqueness of the strong solution were given in \cite[Theorem 2]{Graczyk2019}. Moreover, the necessary and sufficient conditions for the existence of non-negative solutions were provided in \cite[Theorem 3]{Graczyk2019}.

\section{Other matrix models and related particle systems driven by Brownian motion}\label{Sec:otherModels}

In \cite{Bru91}, M.-F. Bru generalized her Wishart process to the following symmetric matrix valued process.

\begin{theorem}[\cite{Bru91}, Theorem 2'] \label{Thm-general Wishart eigenvalue SDE}
	Let $X^N(0)$ be a symmetric non-negative definite deterministic $N \times N$ matrix with distinct eigenvalues. Let $W(t)$ be a matrix Brownian motion, then for $p, \beta, \gamma \in \bR$, the following matrix SDE
	\begin{align} \label{eq-Wishart SDE general}
		dX^N(t) = \gamma \left( \sqrt{X^N(t)} dW(t) + dW(t)^\intercal \sqrt{X^N(t)} \right) + 2 \beta X^N(t) dt + p\gamma^2 I_N dt,
	\end{align}
	has a unique weak solution in the set of symmetric $N \times N$ matrices if $p \in (N-1,N+1)$, and has a unique strong solution that is symmetric positive-definite if $p \ge N+1$.
	
	The ordered eigenvalue processes $\lambda_1^N(t) \ge \cdots \ge \lambda_N^N(t)$ of the unique solution never collide almost surely, and satisfy the following system of SDEs
	\begin{align} \label{eq-SDE eigenvalue general Wishart}
		d\lambda_i^N(t)
		= 2 \gamma \sqrt{\lambda_i^N(t)} dW_i(t)
		+ \left( p \gamma^2 + 2\beta \lambda_i^N(t) + \gamma^2 \sum_{j:j\neq i} \dfrac{\lambda_i^N(t) + \lambda_j^N(t)}{\lambda_i^N(t) - \lambda_j^N(t)} \right) dt, ~ i \in\{ 1, \ldots, N\},
	\end{align}
	where $W_1(t), \ldots, W_N(t)$ are independent standard Brownian motions.
	
	If $p \ge N+1$, $\lambda_N^N(t)>0$ for all $t>0$ almost surely. Furthermore, if $p \in \{1, 2, \ldots, N-1\}$, then the same results hold for the largest $p$ eigenvalue processes $\lambda_1^N(t), \ldots, \lambda_p^N(t)$.
\end{theorem}

\begin{remark}
	\begin{enumerate}
		\item The matrix model \eqref{eq-Wishart SDE general} reduces to the Wishart process \eqref{eq-Wishart matrix SDE} when $\beta=0$ and $\gamma=1$.
		\item When $p \in \bZ$ and $\gamma=1$, the system of SDEs for eigenvalue processes was derived in \cite{Kendall1990} in shape theory.
		\item The system of SDEs \eqref{eq-SDE eigenvalue general Wishart} reduces to \eqref{eq-SDE-Laguerre eigenvalue} when $\beta=0$ and $\gamma^2=2$.
	\end{enumerate}
\end{remark}

For $\beta, \gamma \ge 0$, the singular value of $X^N(t)$ given by \eqref{eq-Wishart SDE general} was studied in \cite{kahn2020}. Let $s_i^N(t) = \sqrt{\lambda_i^N(t)}, i \in\{ 1, \ldots, N\}$ be the singular value processes of $X^N(t)$. Then
\begin{align*}
	ds_i^N(t) = \gamma dW_i(t) + \left( \dfrac{(p-1) \gamma^2}{2 s_i^N(t)} - \beta s_i^N(t) + \dfrac{\gamma^2}{2 s_i^N(t)} \sum_{j:j\neq i} \dfrac{ \big(s_i^N(t)\big)^2 + \big(s_j^N(t)\big)^2}{\big(s_i^N(t)\big)^2 - \big(s_j^N(t)\big)^2} \right) dt.
\end{align*} 
It was obtained in \cite[Theorem 1.1]{kahn2020} the convergence in probability of the sequence of empirical measure processes
\begin{align*}
	\widetilde{L}_N(t)(dx) = \dfrac{1}{2N} \sum_{i=1}^N \Big( \delta_{s_i^N(t)/\sqrt{p}} (dx) + \delta_{-s_i^N(t)/\sqrt{p}} (dx) \Big)
\end{align*}
under general initial conditions. Moreover, the long time behavior of the empirical measure process $\{\lambda_i^N(t)/p, 1 \le i \le N\}$ was also characterized in \cite[Theorem 1.2]{kahn2020}.

A more general class of real symmetric matrix valued processes was introduced in \cite{Graczyk2013}, which is the solution to
\begin{align} \label{SDE-matrix}
	dX_t^N = g_N(X_t^N) dB_t h_N(X_t^N) + h_N(X_t^N) dB_t^{\intercal}
	g_N(X_t^N) + b_N(X_t^N) dt, \quad t\ge 0,
\end{align}
in the space of real symmetric $N \times N$ matrices. Here, $B_t$ is a
$N \times N$ matrix Brownian motion, and the functions $g_N, h_N, b_N
: \bR \rightarrow \bR$ act on the spectrum of $X_t^N$.
(Note. For a real-valued function $f$ and a real symmetric (or complex Hermitian) matrix $X$ that has spectral decomposition $X=\sum_{j=1}^N \alpha_j u_ju_j^*$ with eigenvalues $\{\alpha_j\}_{1 \le j \le N}$ and eigenvectors $\{u_j\}_{1 \le j \le N}$, $f(X)=\sum_{j=1}^N f( \alpha_j) u_ju_j^*$ is the matrix obtained by acting $f$ on the spectrum of $X$.)

The symmetric matrix valued process \eqref{SDE-matrix} extends the previous matrix models in the following aspects: 
\begin{enumerate}
	\item  If we take $g_N(x) = \frac{1}{\sqrt{2N}}$, $h_N(x) = 1$ and $b_N(x) = 0$ in \eqref{SDE-matrix}, then $X^N$ becomes the real symmetric matrix Brownian motion $X^{N,\beta}$ considered in Theorem~\ref{Thm-Dyson SDE}..
	
	\item  If we take $g_N(x) = \sqrt{\frac{\alpha}{2N}}$, $h_N(x) = 1$ and $b_N(x) = -\theta x$ in \eqref{SDE-matrix}, then $X^N$ becomes the real symmetric matrix given in \eqref{eq-matrix SDE general OU}. In particular, if $g_N(x) = \frac{1}{2\sqrt{N}}$, $h_N(x) = 1$ and $b_N(x) = -\frac{1}{2}x$, it is the real symmetric matrix OU process given in \eqref{eq-matrix SDE OU}.
	
	\item  If we take $g_N(x) = \sqrt{x}$, $h_N(x) = \gamma$ and $b_N(x) = 2 \beta x + p\gamma^2$ in \eqref{SDE-matrix}, then $X^N$ becomes the real symmetric matrix given in \eqref{eq-Wishart SDE general}. In particular, if $g_N(x) = \sqrt{x}$, $h_N(x) = 1/\sqrt{N}$, and $b_N(x) =p/N$, then the random matrix $Y^N=NX^N$ is the Wishart process $B^\intercal B$, where $B$ is a $p \times N$ Brownian matrix.
	
	\item  If we take $g_N(x) = \frac{1}{\sqrt{2N}}$, $h_N(x) = 1$ and $b_N(x) = - \frac{1}{2} V'(x)$ in \eqref{SDE-matrix}, then $X^N$ becomes the real symmetric matrix given in \eqref{eq-matrix SDE Li Xiang-Dong}.
	
	\item If we take $g_N(x) = \sqrt{x}$, $h_N(x) = \sqrt{1-x}$ and $b_N(x) = q - (q+r)x$ in \eqref{SDE-matrix} with $q,r > p-1$, then $X^N$ becomes matrix Jacobi processes. See \cite[(4.4)]{Graczyk2013}.
\end{enumerate}

In \cite{Graczyk2013}, the non-collision property of the eigenvalue
processes was established and the system of SDEs for ordered
eigenvalue processes was derived. The results are presented below
where   
$\cC^{1,1}(\bR) = \{f \in \cC^1(\bR): |f'(x) - f'(y)|/|x-y| < \infty\}$.

\begin{theorem}[\cite{Graczyk2013}, Theorems 3 and 5] \label{Thm-Graczyk SDE}
	Let $X_t^N$ be a real symmetric matrix valued stochastic process that solves \eqref{SDE-matrix}. Let $\lambda_1^N(t) \ge \cdots \ge \lambda_N^N(t)$ be the ordered eigenvalue processes of $X_t^N$ and denote the first collision time by
	\begin{align*}
		\tau_N = \inf\{t>0: \exists \ i \neq j, ~\lambda_i(t) = \lambda_j(t) \}.
	\end{align*}
	Suppose that the functions $b_N(x), g_N^2(x), h_N^2(x)$ are
    Lipschitz continuous. Besides, assume that $g_N^2(x) h_N^2(x)$ is
    convex or in $\cC^{1,1}(\bR)$. Furthermore, we assume $\lambda_1^N(0) > \cdots > \lambda_N^N(0)$.
	
	Then we have
    ~$	\bP \left( \tau_N = + \infty \right) = 1$.
	Moreover, the eigenvalue processes satisfy the following SDEs: for $1\le i \le N$,
	\begin{align} \label{SDE-eigenvalue}
		d\lambda_i^N(t)
		&= 2g_N(\lambda_i^N(t)) h_N(\lambda_i^N(t)) dW_i(t) + \left( b_N(\lambda_i^N(t)) + \sum_{j:j\neq i} \dfrac{G_N(\lambda_i^N(t), \lambda_j^N(t))}{\lambda_i^N(t) - \lambda_j^N(t)} \right) dt,
	\end{align}
	where $\{W_i(t)\}_{1\le i\le N}$ are independent Brownian motions and
	\begin{equation}\label{eq-Gn}
		G_N(x,y) = g_N^2(x) h_N^2(y) + g_N^2(y) h_N^2(x).
	\end{equation}
\end{theorem}

\begin{remark}
	Similar results hold for the complex version of \eqref{SDE-matrix}. Namely, under the same conditions in Theorem \ref{Thm-Graczyk SDE}, the eigenvalue processes of the complex Hermitian matrix that solves the matrix SDE
	\begin{align*}
		dX_t^N = g_N(X_t^N) dW_t h_N(X_t^N) + h_N(X_t^N) dW_t^*
		g_N(X_t^N) + b_N(X_t^N) dt, \quad t\ge 0,
	\end{align*}
	where $W_t$ is a complex $N \times N$ matrix Brownian motion, never collide almost surely and satisfy the following system of SDEs
	\begin{align} \label{SDE eigenvalue complex}
		d\lambda_i^N(t)
		&= 2g_N(\lambda_i^N(t)) h_N(\lambda_i^N(t)) dW_i(t) + \left( b_N(\lambda_i^N(t)) + 2 \sum_{j:j\neq i} \dfrac{G_N(\lambda_i^N(t), \lambda_j^N(t))}{\lambda_i^N(t) - \lambda_j^N(t)} \right) dt,
	\end{align}
	where  $\{W_i(t)\}_{1\le i\le N}$ are independent Brownian motions.
\end{remark}

Theorem \ref{Thm-Graczyk SDE} (and its complex analogous) can be proved following the idea used in the proof of Theorem \ref{Thm-Dyson SDE}: the SDEs \eqref{SDE-eigenvalue} (and \eqref{SDE eigenvalue complex}) for eigenvalue processes can be derived by It\^{o} calculus and martingale theory, and the almost sure non-collision of the eigenvalue processes can be proved by the McKean's argument.

\begin{remark}
	The system of SDEs for eigenvector processes were also derived in \cite[Theorem 3]{Graczyk2013}. It was also shown in \cite[Corollary 3]{Graczyk2013} that the system of SDEs for eigenvalue processes and eigenvector processes admits a unique strong solution on $[0,\infty)$ if assuming that $G_N(x,y)$ is strictly positive on the set $\{(x,y) \in \bR^2: x \neq y\}$ together with all conditions in Theorem \ref{Thm-Graczyk SDE}. Another set of conditions for the existence and uniqueness of strong solution (before colliding/exploding) can be found in \cite{Song2019}. However, whether the pathwise uniqueness holds for the matrix SDE \eqref{SDE-matrix} is still unknown.
\end{remark}

Let $L_N(t)$ be the empirical measure process of the eigenvalue processes $\{\lambda_i^N(t)\}_ {1\le i\le N}$ of the symmetric matrix-valued processes $X^{N}(t)$ given in \eqref{SDE-matrix}, that is 
\begin{align*}
	L_N(t) (dx) = \dfrac{1}{N} \sum_{i=1}^N \delta_{\lambda_i^N(t)} (dx).
\end{align*}
The almost sure compactness of the sequence $\{L_N(t), t\in[0,T]\}_{N\in \mathbb N}$ was obtained in \cite{Song2019} by using the compactness argument presented in the proof of Theorem \ref{Thm-Wigner dynamic-strong}, and the equation for the limit measures was derived as well. Note that similar problems were also investigated in \cite{Malecki-arxiv} independently.

\begin{theorem}[\cite{Song2019}, Theorems 2.1 and 2.2] \label{Thm-Graczyk matrix}
	Let $T>0$ be a fixed number. Suppose that \eqref{SDE-eigenvalue} has a strong solution which does not explode or collide for $t \in [0,T]$. Assume the following conditions hold:
	\begin{enumerate}
		\item There exists a positive function $\varphi(x) \in C^2(\mathbb{R})$ such that
		$\lim\limits_{|x|\rightarrow +\infty} \varphi(x) = +\infty,$
		$\varphi'(x)b_N(x)$ is bounded with respect to $(x,N)$, and
          $\varphi'(x)g_N(x)h_N(x)$ satisfies, 		for some positive integer $l_1$, 
		\begin{align*}
			\sum_{N=1}^{\infty} \left( \dfrac{\|\varphi' g_N h_N\|_{L^{\infty}(dx)}^2}{N} \right)^{l_1} < \infty.
		\end{align*}
		
		\item The function
		$N G_N(x,y) \dfrac{\varphi'(x) - \varphi'(y)}{x-y}$
		is bounded with respect to $(x,y, N)$. 
		
		\item The empirical measure $L_N(0)$ converges weakly to a measure $\mu_0$ as $N$ goes to infinity almost surely, and
		\begin{align}\label{eq-C0}
			C_0 = \sup_{N>0} \langle \varphi, L_N(0) \rangle
			= \sup_{N>0} \dfrac{1}{N} \sum_{i=1}^N \varphi \left( \lambda_i^N(0) \right) < \infty.
		\end{align}
	
		\item There exists a sequence $\{\tilde{f}_k\}_{k\in \mathbb
          N}$ of $C^2(\bR)$  functions such that it is dense in the
          space $C_0(\mathbb{R})$ of continuous functions vanishing at
          infinity and  that $\tilde{f}_k'(x) g_N(x) h_N(x)$ satisfies,
          for some positive integer $l_2 \ge 2$,
		\begin{align}\label{eq-psi}
			\psi(k) = \sum_{N=1}^{\infty} \left( \dfrac{\|\tilde{f}_k' g_N h_N\|_{L^{\infty}(dx)}^2}{N} \right)^{l_2} < \infty.
		\end{align}

		\item There exist continuous functions $b(x)$ and $G(x,y)$, such that $b_N(x)$ converges to $b(x)$ and $N G_N(x,y)$ converges to $G(x,y)$ uniformly as $N$ tends to infinity.
	\end{enumerate}
	
	Then the sequence $\{L_N(t), t\in[0,T]\}_{N\in \mathbb N}$ is relatively compact in $C([0,T], M_1(\mathbb{R}))$ almost surely, i.e., every subsequence has a further subsequence that converges in $C([0,T], M_1(\mathbb{R}))$ almost surely. Furthermore, any limit measure $\mu$ in $C([0,T], M_1(\mathbb{R}))$ satisfies the equation
	\begin{align} \label{eq-limit equation Graczyk matrix}
		\langle f, \mu_t \rangle
		=& \langle f, \mu_0 \rangle + \int_{0}^t \langle bf', \mu_s \rangle ds
		+ \dfrac{1}{2} \int_{0}^t \left[ \iint_{\bR^2} \dfrac{f'(x) - f'(y)}{x-y} G(x,y) \mu_s(dx) \mu_s(dy) \right] ds,
	\end{align}
	for all $f \in C_b^2(\bR)$ such that $f'(x)b(x)$ and $\frac{f'(x)-f'(y)}{x-y} G(x,y)$ are bounded.
\end{theorem}

\begin{remark}
	The equation \eqref{eq-limit equation Graczyk matrix} with the test function $f = (z-x)^{-1}$ for $z \in \bC_+$ was derived in \cite{Song2019}. Indeed, the computation therein is valid for all $f \in C_b^2(\bR)$.
\end{remark}

By using the symmetric polynomials and the tightness argument presented in the proof of Theorem \ref{Thm-Wigner dynamic-weak}, \cite{Malecki-arxiv} obtained the tightness of the sequence $\{L_N(t), t\in[0,T]\}_{N\in \mathbb N}$ for both the real case \eqref{SDE-eigenvalue} and the complex case \eqref{SDE eigenvalue complex} with general test functions in $C_b^2(\bR)$. The equation for the limit measures in law was also derived. The results are presented below.

\begin{theorem}[\cite{Malecki-arxiv}, Theorem 1] \label{Thm-Graczyk model weak}
	Assume that $g_N$, $h_N$ and $b_N$ are continuous and satisfy
	\begin{align*}
		g_N^2(x) + h_N^2(x) \le K(1+|x|), \
		|b_N(x)| \le KN (1 + |x|), \
		\forall x \in \bR, \ \forall N \in \bN,
	\end{align*}
	for some positive constant $K$. Suppose that
	\begin{align*}
		\sup_{N\in\bN} \int_{\bR} x^8 L_N(0)(dx) < \infty,
	\end{align*}
	then the sequence of the measure-valued processes $\{L_N(t), t\in[0,T]\}_{N\in \mathbb N}$ related to \eqref{SDE-eigenvalue} (resp. \eqref{SDE eigenvalue complex}) is tight. Furthermore, assuming that $g_N^2(x) \to g^2(x)$, $h_N^2(x) \to h^2(x)$ and $b_N(x)/N \to b(x)$ locally uniformly on $\bR$ as $N \to \infty$, then any limit measure $\mu$ of a weakly convergent subsequence in law is an element in $C([0,T], \cP(\bR))$ that satisfies
	\begin{align*}
		\langle f, \mu_t \rangle
		=& \langle f, \mu_0 \rangle + \int_{0}^t \langle bf', \mu_s \rangle ds
		+ \dfrac{\beta}{2} \int_{0}^t \left[ \iint_{\bR^2} \dfrac{f'(x) - f'(y)}{x-y} G(x,y) \mu_s(dx) \mu_s(dy) \right] ds,
	\end{align*}
	for all $t > 0$, for all $f \in C_b^2(\bR)$, where $\beta=1$ corresponds to the real case \eqref{SDE-eigenvalue} while $\beta = 2$ corresponds to the complex case \eqref{SDE eigenvalue complex}.
\end{theorem}

\begin{remark}
	It is worth pointing out that the almost sure compactness obtained in Theorem \ref{Thm-Graczyk matrix} is stronger than the tightness established in Theorem \ref{Thm-Graczyk model weak}. However, in comparison with Theorem \ref{Thm-Graczyk matrix}, Theorem \ref{Thm-Graczyk model weak} does not require the non-colliding property of the strong solution to \eqref{SDE-eigenvalue}. Hence, Theorem \ref{Thm-Graczyk model weak} is applicable to the $\beta$ version of Dyson Brownian motion \eqref{eq-Dyson BM SDE} with $\beta \in (0, \infty)$, while Theorem \ref{Thm-Graczyk matrix} is only valid for $\beta \in [1, \infty)$.
\end{remark}

The system \eqref{SDE-eigenvalue} for eigenvalues of matrix-valued process \eqref{SDE-matrix} was further generalized in \cite{Graczyk2014} to the following particle system: for  $1 \le i \le N$,
\begin{align} \label{SDE-particle}
	\begin{cases}
		& dx_i^N(t) = \sigma_i^N(x_i^N(t)) dW_i(t)+ \left( b_i^N(x_i^N(t)) + \sum_{j:j\neq i} \dfrac{H_{ij}^N(x_i^N(t), x_j^N(t))}{x_i^N(t) - x_j^N(t)} \right) dt, \\
		& x_1(t) \le \cdots \le x_N(t), ~t \ge 0,
	\end{cases}
\end{align}
where $\{W_i(t)\}_{1 \le i \le N}$ is a family of independent Brownian motions.

\begin{remark}
	\begin{itemize}
		\item If we take $\sigma_i^N(x) = \sigma_N(x), b_i^N(x) = b_N(x)$ and $H_{ij}^N(x,y) = \gamma_N$, then the particle system \eqref{SDE-particle} reduces to the system \eqref{eq-SDE particle Cepa97}.
		
		\item If we take $\sigma_i^N(x) = 2x, b_i^N(x) = (N+1)x$ and $H_{ij}^N(x,y) = 2xy$, then the particle system \eqref{SDE-particle} reduces to the system \eqref{eq-eigenvalue SDE ellipsoids}. To see this, we apply It\^{o}'s formula to \eqref{eq-eigenvalue SDE ellipsoids} to obtain
		  \begin{align*}
		 &   d \lambda_i^N(t) = d \left( e^{\ln \lambda_i^N(t)} \right)
			= e^{\ln \lambda_i^N(t)} d \left( \ln \lambda_i^N(t) \right)
			+ \dfrac{1}{2} e^{\ln \lambda_i^N(t)} d \langle \ln \lambda_i^N(t) \rangle \\
			=~& 2\lambda_i^N(t) dW_i(t) + \lambda_i^N(t) \sum_{j:j\neq i} \dfrac{\lambda_i^N(t) + \lambda_j^N(t)}{\lambda_i^N(t) - \lambda_j^N(t)} dt + 2 \lambda_i^N(t) dt 
			= 2\lambda_i^N(t) dW_i(t) + \left( (N+1) \lambda_i^N(t) + \sum_{j:j\neq i} \dfrac{2\lambda_i^N(t) \lambda_j^N(t)}{\lambda_i^N(t) - \lambda_j^N(t)} \right) dt.
		\end{align*}
	\end{itemize}
\end{remark}

The following theorem guarantees the existence and uniqueness of the strong non-exploding and non-colliding solution to \eqref{SDE-particle}

\begin{theorem}[\cite{Graczyk2014}, Theorem 2.2] \label{Thm-Graczyk particle}
	Consider the system \eqref{SDE-particle} with initial condition $x_1(0) \le \cdots \le x_N(0)$. Assume the following conditions hold:
	\begin{enumerate}
		\item The coefficient functions $\sigma_i^N(x)$, $b_i^N(x)$ are continuous for $1 \le i \le N$ while $H_{ij}^N(x,y)$ is non-negative, continuous and satisfies the symmetric condition $H_{ij}^N(x,y) = H_{ji}^N(y,x)$ for $1\le i \neq j \le N$.
		\item There exists a function $\rho: \bR_+ \to \bR_+$ satisfying $\int_{0+} \rho^{-1}(x) dx = \infty$, such that for $1 \le i \le N$,
		\begin{align*}
			\left| \sigma_i^N(x) - \sigma_i^N(y) \right|^2 \le \rho(|x-y|), ~ \forall x, y \in \bR.
		\end{align*}
	
		\item There exists a positive constant $C$ that may depends on $N$, such that for all $1 \le i \neq j \le N$,
          \[
			\sigma_i^N(x)^2 + x b_i^N(x) \le C(1+x^2),
			~ \forall x \in \bR; \qquad
			H_{ij}^N(x,y) \le C(1+|xy|),
			~ \forall x, y \in \bR.
            \]
	
		\item For $1 \le i \neq j \le N$,
		\begin{align*}
			\dfrac{H_{ij}^N(w,z)}{z-w}
			\le \dfrac{H_{ij}^N(x,y)}{y-x},
			~ \forall w < x < y < z.
		\end{align*}
	
		\item There exists a positive constant $C$ that may depends on $N$, such that for all $1 \le i \neq j \le N$,
		\begin{align*}
			\sigma_i^N(x)^2 + \sigma_j^N(y)^2
			\le C(x-y)^2 + 4 H_{ij}^N(x,y),
			~ \forall x,y \in \bR.
		\end{align*}
	
		\item There exists a positive constant $C$ that may depends on $N$, such that for all $1 \le i < j < k \le N$, for all $x<y<z$,
		\begin{align*}
			(y-x) H_{ij}^N(x,y) + (z-y) H_{jk}^N(y,z)
			\le C(z-y)(z-x)(y-x) + (z-x) H_{ik}^N(x,z).
		\end{align*}
	
		\item For $1 \le k < l \le N$, the set
          $\displaystyle
		  G_{kl} = \bigcap_{k<i<j<l} \left\{ x \in \bR: \sigma_i^N(x)^2 + \sigma_j^N(x)^2 + H_{ij}^N(x,x) = 0 \right\}
          $
		consists of isolated points and for every $x \in G_{kl}$,
		\begin{align*}
			\sum_{i=k}^l \left( b_i^N(x) + \sum_{j=1}^{N-2} \dfrac{H_{ij}^N(x,y_j)}{x-y_j} \text{1}_{\bR \setminus \{x\}}(y_j) \right) \neq 0,
			~ \forall y_1, \ldots, y_{N-2} \in \bR.
		\end{align*}
	
		\item The function $b_i^N(x)$ is Lipschitz continuous or non-increasing for $1 \le i \le N$. Moreover, for $1 \le i < j \le N$, for all $x \in \bR$, $b_i^N(x) \le b_j^N(x)$.
	\end{enumerate}

	Then there exists a unique strong non-exploding solution of \eqref{SDE-particle}, such that the first collision time
	\begin{align*}
		\tau_N = \inf\{t>0: \exists \ i \neq j, ~x_i^N(t) = x_j^N(t) \}
	\end{align*}
	is infinite almost surely.
\end{theorem}

\begin{remark}
	The initial values for the particles in Theorem \ref{Thm-Graczyk particle} are allowed to collide.
\end{remark}

Let $L_N(t)$ be the empirical measure process of the particles $\{x_i^N(t)\}_ {1\le i\le N}$ given in \eqref{SDE-particle}, that is 
\begin{align*}
	L_N(t) (dx) = \dfrac{1}{N} \sum_{i=1}^N \delta_{x_i^N(t)} (dx).
\end{align*}
The convergence of the sequence $\{L_N(t), t\in[0,T]\}_{N\in \mathbb N}$ for $T>0$ was studied in \cite{Song2019} for the case that the family of functions $\{b_i^N(x)\}_{1 \le i \le N}$ and $\{H_{ij}^N(x,y)\}_{1 \le i \neq j \le N}$ are identical respectively. For simplicity, we assume $b_i^N(x) = b_N(x)$, $\sigma_i^N(x) = \sigma_N(x)$ for all $1 \le i \le N$ and $H_{ij}^N(x,y) = H_N(x,y)$ for $1 \le i \neq j \le N$, and then the particle system \eqref{SDE-particle} becomes
\begin{align} \label{SDE-particle'}
	dx_i^N(t) = \sigma^N(x_i^N(t)) dW_i(t)+ \left( b_N(x_i^N(t)) + \sum_{j:j\neq i} \dfrac{H_N(x_i^N(t), x_j^N(t))}{x_i^N(t) - x_j^N(t)} \right) dt, \ t \ge 0,
\end{align}
for $1 \le i \le N$.

\begin{theorem}[\cite{Song2019}, Theorems 3.1 and 3.2] \label{Thm-LSD-Graczyk particle}
	Let $T>0$ be a fixed number. Suppose that \eqref{SDE-particle'} has a strong solution that is non-exploding and non-colliding for $t \in [0,T]$. Assume the following conditions hold:
	\begin{enumerate}
		\item There exists a positive function $\varphi(x) \in C^2(\mathbb{R})$ such that
		  $\lim\limits_{|x|\rightarrow +\infty} \varphi(x) = +\infty$,
          $\varphi'(x)b_N(x)$ and $\varphi''(x) \sigma^N(x)^2$ are
          bounded with respect to $(x,N)$, and $\varphi'(x)    \sigma^N(x)$ satisfies,
          for some positive integer $l_1$.
		  \begin{align*}
			\sum_{N=1}^{\infty} \left( \dfrac{\|\varphi' \sigma^N\|_{L^{\infty}(dx)}^2}{N} \right)^{l_1} < \infty.
		\end{align*}
		
		\item The function
		$N H_N(x,y) \dfrac{\varphi'(x) - \varphi'(y)}{x-y}$
		is bounded with respect to $(x,y, N)$. 
		
		\item The empirical measure $L_N(0)$ converges weakly to a measure $\mu_0$ as $N$ goes to infinity almost surely, and
		\begin{align*}
			C_0 = \sup_{N>0} \langle \varphi, L_N(0) \rangle
			= \sup_{N>0} \dfrac{1}{N} \sum_{i=1}^N \varphi \left( x_i^N(0) \right) < \infty.
		\end{align*}
		
		\item There exists a sequence $\{\tilde{f}_k\}_{k\in \mathbb
          N}$ of $C^2(\bR)$  functions such that it is dense in the
          space $C_0(\mathbb{R})$ of continuous functions vanishing at
          infinity and  that $\tilde{f}_k'(x) \sigma^N(x)$ satisfies, 		for some positive integer $l_2 \ge 2$,
		\begin{align*}
			\psi(k) = \sum_{N=1}^{\infty} \left( \dfrac{\|\tilde{f}_k' \sigma^N\|_{L^{\infty}(dx)}^2}{N} \right)^{l_2} < \infty.
		\end{align*}
		
		\item There exist continuous functions $b(x)$, $\sigma(x)$ and $H(x,y)$, such that $b_N(x)$ converges to $b(x)$, $\sigma^N(x)$ converges to $\sigma(x)$ and $N H_N(x,y)$ converges to $H(x,y)$ uniformly as $N$ tends to infinity.
	\end{enumerate}
	
	Then the sequence $\{L_N(t), t\in[0,T]\}_{N\in \mathbb N}$ is
    relatively compact in  $C([0,T], M_1(\mathbb{R}))$, i.e., every
    subsequence has a further subsequence that converges in
    $C([0,T], M_1(\mathbb{R}))$ almost surely. Furthermore, any limit
    measure $\mu$ in \linebreak $C([0,T], M_1(\mathbb{R}))$ satisfies the equation
	\begin{align} \label{eq-limit equation Graczyk particle}
		\langle f, \mu_t \rangle
		=& \langle f, \mu_0 \rangle + \int_{0}^t \langle f'b, \mu_s \rangle ds
		+ \dfrac{1}{2} \int_0^t \langle f'' \sigma^2, \mu_s \rangle ds 
		+ \dfrac{1}{2} \int_{0}^t \left[ \iint_{\bR^2} \dfrac{f'(x) - f'(y)}{x-y} H(x,y) \mu_s(dx) \mu_s(dy) \right] ds,
	\end{align}
	for all $f \in C_b^2(\bR)$ such that $f'(x)b(x)$, $f'(x)
    \sigma(x)$, $f''(x) (\sigma(x))^2$ and $\frac{f'(x)-f'(y)}{x-y}
    H(x,y)$ are bounded as well as \linebreak  $\| f''(\cdot)(\sigma^N(\cdot))^2 - f''(\cdot)(\sigma(\cdot))^2 \|_{L^{\infty}} \to 0$ as $N \to \infty$.
\end{theorem}	

\begin{remark}
	The equation \eqref{eq-limit equation Graczyk particle} for limit measures with the test function $f = (z-x)^{-1}$ for $z \in \bC_+$ was derived in \cite{Song2019}. Indeed, the computation there is valid for all $f \in C_b^2(\bR)$.
\end{remark}

We would like to point out that the conditions for the uniqueness of the solution to \eqref{eq-limit equation Graczyk matrix} and \eqref{eq-limit equation Graczyk particle} are still unknown. The fluctuations of the sequence $\{L_N(t), t\in[0,T]\}_{N\in \mathbb N}$ around its limits was studied in \cite{Song19}. We refer the CLT in \cite[Theorem 2.1]{Song19} for details.

There is a huge literature on related interacting particle systems, particularly on those related to Bessel processes. For more details, we refer to the survey papers \cite{SurveyBessel03, Zambotti17} and the recent book \cite{Katori16}.

\section{Matrix-valued stochastic processes driven by fractional Brownian motion}
\label{Sec:fBm}

A common feature of the matrix-valued stochastic processes discussed so far is that they are all driven by independent Brownian motions. In contrast, the study of matrix-valued SDEs driven by fractional Brownian motions has a shorter history and is relatively limited.

Recall that a centred Gaussian process $B = \{B(t), t \ge 0\}$ is called fractional Brownian motion with Hurst parameter $H \in (0,1)$ if it has the covariance function
\begin{align*}
	\bE[B(t)B(s)] = \dfrac{1}{2} \big( s^{2H} + t^{2H} - |t-s|^{2H} \big).
\end{align*}
We refer the reader to \cite{Nualart2006} for more details.

To our best knowledge, the first paper in this area is \cite{Nualart20144266}, where the real symmetric matrix fractional Brownian motion was introduced and studied.

\begin{definition} \label{Def-matrix fBm}
  Let $\{B_{i,j}(t), 1 \le i \le j \le N\}$ be a family of i.i.d. fractional Brownian motions with Hurst parameter $H \in (0,1)$. Let $H^N(t) = \left( H_{k,l}^N(t) \right)_{1 \le k \le l \le N}$ be a real symmetric $N \times N$ matrix-valued process whose entries are
\[
H_{k,l}^N(t) =
B_{k,l}(t)  {\large 1}_{\{ k < l\}}
+		\sqrt{2} B_{l,l}(t)  {\large 1}_{\{ k = l\}}.
\]
Then $B(t)$ is called the real symmetric matrix fractional Brownian motion with Hurst parameter $H$.
\end{definition}

It is natural to consider the eigenvalue processes as we have done for the matrix Brownian motion. For $1 \le i \le N$, denote by $\varPhi_i$ the function that maps a $N \times N$ real symmetric matrix to its $i$-th largest eigenvalue. The following results can be found in \cite{Nualart20144266}.

\begin{theorem}[\cite{Nualart20144266}, Theorems 4.1 and  5.2] \label{Thm-eigenvalue SDE matrix fBm}
	Let $X^N(0)$ be a real symmetric $N \times N$ deterministic matrix and let $X^N(t) = X^N(0) + H^N(t)$, where $H^N(t)$ is defined in Definition \ref{Def-matrix fBm} with Hurst parameter $H \in (1/2,1)$. Let $\lambda_1^N(t) \ge \lambda_2^N(t) \ge \cdots \ge \lambda_N^N(t)$ be the ordered eigenvalue processes of $X^N(t)$. Denote the first collision time of the eigenvalue processes by
	\begin{align*}
		\tau_N = \inf \left\{ t>0: \exists \ i \neq j, ~\lambda_i^N(t) = \lambda_j^N(t) \right\}.
	\end{align*}
	Then
    $		\bP \left( \tau_N = + \infty \right) = 1$.
	Furthermore, the ordered eigenvalue processes $\lambda_1^N(t), \cdots, \lambda_N^N(t)$ satisfy
	\begin{align} \label{eq-SDE-fBm eigenvalue}
		\lambda_i^N(t)
		= \lambda_i^N(0)
		+ \left( \sum_{k \le l} \int_0^t \dfrac{\partial \varPhi_i(X^N(s))}{\partial B_{k,h}} \delta B_{k,h}(s) \right)
		+ 2H \sum_{j:j\neq i} \int_0^t \dfrac{s^{2H-1}}{\lambda_i^N(s) - \lambda_j^N(s)} ds.
	\end{align}
\end{theorem}

The main tool used in \cite{Nualart20144266} is the fractional calculus and Malliavin calculus, for which we refer to \cite{skm93} and \cite{Nualart2006} respectively. The non-collision property was obtained in \cite[Theorem 4.1]{Nualart20144266} by establishing an upper bound for negative moments of the difference of eigenvalue processes. The equation \eqref{eq-SDE-fBm eigenvalue} for eigenvalue processes was derived in \cite[Theorem 5.2]{Nualart20144266} by employing a multidimensional version of the It\^{o}'s formula for the Skorohod integral (\cite[Theorem 3.1]{Nualart20144266}, see also \cite[Theorem 2]{Pardo2017}).

\begin{remark}
	The complex version of Theorem \ref{Thm-eigenvalue SDE matrix fBm} can be obtained by the same argument.
\end{remark}

\begin{remark}
	Unlike the Brownian motion case ($H=1/2$), the L\'{e}vy characterization theorem for fractional Brownian motion \cite[Theorem 3.1]{Hu2009} is not applicable here, and it is conjectured that the second term in the right hand side of \eqref{eq-SDE-fBm eigenvalue} is even not be Gaussian, see \cite[Remark 5.3]{Nualart20144266}.
\end{remark}

For the sequence of eigenvalue processes $\{\lambda_i^N(t)\}_{1\ le i \le N}$ in Theorem \ref{Thm-eigenvalue SDE matrix fBm}, denote the empirical measure of the normalized eigenvalue processes by
\begin{align*}
	L_N(t)(dx) = \dfrac{1}{N} \sum_{i=1}^N \delta_{\lambda_i^N(t)/\sqrt{N}}(dx).
\end{align*}
For the case $H \in (1/2,1)$, the convergence in probability of the sequence $\{L_N(t)\}_{N \in \bN}$ to the semi-circle law was established in \cite{Pardo2016} under null initial condition $X^N(0) = 0$ by using Malliavin calculus and the tightness argument used in the proof of Theorem \ref{Thm-Wigner dynamic-weak}.

Similar to the free Brownian motion, the non-commutative fractional Brownian motion with Hurst parameter $H \in (0,1)$ was introduced in \cite[Definition 3.1]{Nourdin2014} as a centred semicircular process $\{S^H(t), t \ge 0\}$ with covariance function
\begin{align*}
	\tau \big( S^H(t) S^H(s) \big)
	= \dfrac{1}{2} \big( t^{2H} + s^{2H} - |t-s|^{2H} \big),
\end{align*}
where $\tau$ is the trace on the non-commutative probability space. The semi-circle law $\{\mu_t^H, t \ge 0\}$ established in \cite{Pardo2016} is the law of a non-commutative fractional Brownian motion with Hurst parameter $H$.

The results of \cite{Pardo2016} were extended to normalized real symmetric matrix Gaussian processes with general initial condition in \cite{Jaramillo2019}. In particular, the real symmetric matrix Gaussian processes considered in \cite{Jaramillo2019} include the real symmetric matrix fractional Brownian motion with $H \in (0,1)$.

The almost sure convergence of the sequence of eigenvalue empirical measure valued processes of Wigner-type matrices, whose entries are generated from the solution of $1$-dimensional Stratonovich SDE
\begin{align} \label{eq-Stratonovich SDE}
	dX_t = \sigma(X_t) \circ dB_t^H + b(X_t) dt	, ~ t \ge 0,
\end{align}
to the semi-circle law was established in \cite[Theorem 3.1, 3.2]{Song2020} for $H \in (1/2,1)$ by using fractional calculus and the argument used in the proof of Theorem \ref{Thm-Wigner dynamic-strong}. \cite{Song2020} also studied the convergence of the sequence of eigenvalue empirical measure-valued processes of the complex analogue and the real symmetric matrix with local dependent entries.

For any test function $f \in C^4(\bR)$ whose fourth derivative has polynomial growth, the random fluctuation
\begin{align*}
	\cL_N^\circ(f)(t) = N \big( \langle f,L_N(t) \rangle - \bE \left[ \langle f,L_N(t) \rangle \right] \big).
\end{align*}
of the sequence $\{L_N(t)\}_{N \in \bN}$ of normalized real symmetric matrix Gaussian process around its expectation was studied in \cite[Theorem 2.3]{Jaramillo2020}.

The fractional version of Wishart process was studied in \cite{Pardo2017} for $H \in (1/2,1)$. Let $B(t)$ be a $N \times p$ matrix whose entries are i.i.d. standard real fractional Brownian motion (matrix fractional Brownian motion) with Hurst parameter $H \in (1/2,1)$. Let $X^N(t) = (B(t)+A) (B(t)+A)^\intercal$, where $A$ is a $N \times p$ real deterministic matrix, is the fractional Wishart process. The following result provides are the non-collision property of eigenvalue processes as well as the equations satisfied by the eigenvalue processes.

\begin{theorem}[\cite{Pardo2017}, Theorems 3 and 4] \label{Thm-eigenvalue SDE fractional Wishart}
	 Let $X^N(t)$ be a fractional Wishart process with $H \in (1/2,1)$. Let $\lambda_1^N(t) \ge \lambda_2^N(t) \ge \cdots \ge \lambda_N^N(t)$ be the ordered eigenvalue processes of $X^N(t)$. Denote the first collision time of the eigenvalue processes by
	\begin{align*}
		\tau_N = \inf \left\{ t>0: \exists \ i \neq j, ~\lambda_i^N(t) = \lambda_j^N(t) \right\}.
	\end{align*}
	Then
    $	\bP \left( \tau_N = + \infty \right) = 1$.
	Furthermore, the ordered eigenvalue processes $\lambda_1^N(t), \ldots, \lambda_N^N(t)$ satisfy
    \[
	\lambda_i^N(t)
	= \lambda_i^N(0)
	+ \left( \sum_{k=1}^N \sum_{l=1}^p \int_0^t \dfrac{\partial \varPhi_i(X^N(s))}{\partial B_{k,l}} \delta B_{k,l}(s) \right) 
		+ 2H \int_0^t \left( p + \sum_{j:j\neq i} \dfrac{\lambda_i^N(s) + \lambda_j^N(s)}{\lambda_i^N(s) - \lambda_j^N(s)} \right) s^{2H-1} ds.
        \]
\end{theorem}

For the sequence of eigenvalue processes $\{\lambda_i^N(t)\}_{1\le i \le N}$ in Theorem \ref{Thm-eigenvalue SDE fractional Wishart}, denote the empirical measure of the normalized eigenvalue processes by
\begin{align*}
	L_N(t)(dx) = \dfrac{1}{N} \sum_{i=1}^N \delta_{\lambda_i^N(t)/N}(dx).
\end{align*}
For the case $H \in (1/2,1)$, the convergence in probability of the sequence $\{L_N(t)\}_{N \in \bN}$ to the Mar\v{c}enko-Pastur law was established in \cite{Pardo2017} under null initial condition $X^N(0) = 0$ by using Malliavin calculus and the tightness argument used in the proof of Theorem \ref{Thm-Wigner dynamic-weak}. As an extension, the convergence in probability of the sequence of eigenvalue empirical measure valued processes of Wishart-type matrices, whose entries are generated from the solution of \eqref{eq-Stratonovich SDE}, to the Mar\v{c}enko-Pastur law was established in \cite[Theorems 3.1, 3.2]{Song2020} by using fractional calculus and the argument used in the proof of Theorem \ref{Thm-Wigner dynamic-strong}.

\section{Matrix-valued stochastic processes driven by Brownian sheet}
\label{Sec:BrownianSheet}

In stochastic analysis, multi-parameter processes, which are also known as random fields, are a natural extension of one-parameter processes. Various interactions exist between the theory of multi-parameter processes and other disciplines, such as analysis, algebra, mathematical statistics and statistical mechanics. The most important multi-parameter process is the Brownian sheet. Recall that the standard $1$-dimensional ($2$-parameter) Brownian sheet $\{B(s,t), (s,t)\in\bR_+^2\}$ is a centred Gaussian random 
field with covariance function
\begin{align*}
	\bE \left[ B(s_1,t_1) B(s_2,t_2) \right] = (s_1 \wedge s_2)  (t_1 \wedge t_2).
\end{align*}
We refer to \cite{Davar2002} for more details on multi-parameter processes, in particular for properties of the Brownian sheet.

As shown in the review, there is a rich literature on eigenvalue processes of matrix-valued one-parameter processes. In contrast, the study on matrix-valued multi-parameter processes is just beginning. To our best knowledge, \cite{SongXiao2021} is the only reference on this topic.

Let $\{B_{i,j}(s,t), s,t \in \bR_+\}_{i,j \ge 1}$ be a family of independent standard $1$-dimensional Brownian sheets. The $N \times N$ real symmetric matrix-valued stochastic process $H^N(s,t) = \left( H_{i,j}^N(s,t) \right)_{1 \le i, j \le N}$ with entries
\begin{align} \label{eq-def Y}
	H_{i,j}^N(s,t) =
	\begin{cases}
		B_{i,j}(s,t), & i<j, \\
		\sqrt{2} B_{i,i}(s,t), & i=j,
	\end{cases}
\end{align}
is known as the real symmetric matrix Brownian sheet. Let $A^N$ be a $N \times N$ real symmetric deterministic matrix with distinct eigenvalues, and let
\begin{align} \label{eq-def-matrix Brownian sheet}
	X^N(s,t) = H^N(s,t) + A^N.
\end{align}

In the case of symmetric matrix Brownian motion, multi-dimensional It\^{o}'s formula for Brownian motions plays a key role when deriving the system of SDEs \eqref{eq-Dyson BM SDE} for eigenvalue processes (see \cite[Theorem 4.3.2]{Anderson2010}). The system of SDEs \eqref{eq-SDE-fBm eigenvalue} for eigenvalue processes of real symmetric matrix fractional Brownian motion also heavily relies on the multi-dimensional It\^{o}'s formula for the Skorohod integral with respect to fractional Brownian motion (see \cite{Nualart20144266}). For the case of Brownian sheet, though the one-dimensional It\^{o}'s formula has been established in \cite{Walsh1975}, the multi-dimensional version was not available until it was obtained in \cite{SongXiao2021} by using the stochastic calculus on the plane developed in \cite{Walsh1975}. Using the multi-dimensional It\^{o}'s formula, \cite{SongXiao2021} derived the system of SPDEs for the ordered eigenvalue processes of $X^N(s,t)$. Moreover, the convergence of the sequence of the eigenvalue empirical measure processes of $\frac{1}{\sqrt{d}} X^N(s,t)$ was also studied in \cite{SongXiao2021}.

\section{Open problems} \label{Sec:Open problems}

For the system of eigenvalue processes \eqref{SDE-eigenvalue}, the conditions of the uniqueness to \eqref{eq-limit equation Graczyk matrix} are still unknown. The uniqueness can be obtained under proper conditions for some special matrix-valued processes. It was established in \cite{Duvillard2001} for real symmetric and complex Hermitian matrix Brownian motion and complex Wishart process using large deviation technique. For real symmetric and complex Hermitian matrix Brownian motion, \cite{Anderson2010} obtained the uniqueness by analyzing the Stieltjes transform of the limit measure process. For real symmetric matrix OU process, the uniqueness was established by computing the moments of limiting measure in \cite{Chan1992}. In \cite{Rogers1993}, the uniqueness was established also by analyzing the Stieltjes transform of the limit measure process. For the Dyson's Brownian motion with general drift, the uniqueness was established in \cite{LLX20} by the entropy technique. In \cite{Cepa1997}, the uniqueness was established by transferring the equation of the Stieltjes transform of the limit measure to a PDE, which was solved via Fourier transform. It was established in \cite{Song2019} the uniqueness of self-similar solution for real symmetric and complex Hermitian matrix Brownian motion, Wishart process and Laguerre process. However, \cite[Proposition 5, 6]{Malecki-arxiv} provided an example of \eqref{eq-limit equation Graczyk matrix} which have at least two solutions. It is also of interest to know the conditions under which the equation \eqref{eq-limit equation Graczyk matrix} has unique solution that is supported on $[0,+\infty)$. Similarly, the conditions for the uniqueness of the limiting equation \eqref{eq-limit equation Graczyk particle} are still unknown.
	
For the fractional Wishart process considered in Theorem \ref{Thm-eigenvalue SDE fractional Wishart}, the fluctuation of the sequence of eigenvalue empirical measure processes around its limiting measure process is still unknown.
	
For the symmetric matrix valued Brownian sheet, the fluctuation of the sequence of eigenvalue empirical measure processes around its limiting measure process is also unknown.

\bibliographystyle{myjmva}
\bibliography{survey}

\end{document}